\documentclass[12pt,a4paper]{amsart}
\usepackage[utf8]{inputenc}
\usepackage{package}

\begin{document}

\title{Wind field reconstruction with adaptive random Fourier features}

\author{Jonas Kiessling} 
\address{H-Ai AB, Stockholm, Sweden and KTH Royal Institute of Technology, Stockholm, Sweden}

\author{Emanuel Ström} 
\address{KTH Royal Institute of Technology, Stockholm, Sweden}

\author{Raúl Tempone}
\address{RWTH Aachen University, Aachen, Germany and KAUST, Saudi Arabia}

\date{January 2021}

\maketitle

\vspace{5cm}
\begin{abstract}
We investigate the use of spatial interpolation methods for reconstructing the horizontal near-surface wind field given a sparse set of measurements. In particular, random Fourier features is compared to a set of benchmark methods including Kriging and Inverse distance weighting. Random Fourier features is a linear model $\beta(\pmb x) = \sum_{k=1}^K \beta_k e^{i\omega_k \pmb x}$ approximating the velocity field, with frequencies $\omega_k$ randomly sampled and amplitudes $\beta_k$ trained to minimize a loss function. Inspired by \cite{tempone99}, we include a physically motivated divergence penalty term $|\nabla \cdot \beta(\pmb x)|^2$, as well as a penalty on the Sobolev norm. We derive a bound on the generalization error and derive a sampling density that minimizes the bound. Following  \cite{kammonen2020}, we devise an adaptive Metropolis-Hastings algorithm for sampling the frequencies of the optimal distribution. In our experiments, our random Fourier features model outperforms the benchmark models.\\[0.5cm]
\end{abstract}
\clearpage\tableofcontents
\clearpage
\section{Introduction}
An integral part in wind farm planning and weather forecasting is access to high quality wind data. The available data is often sparse and interpolation techniques are employed in order to increase spatial resolution. Depending on the application, interpolation can be done with respect to time aggregates or for each measurement. In wind farm site planning for example, a point of interest is the expected energy output over time, which can be estimated as a function of the wind speed. A common approach is to approximate the distribution of the wind speed over time with some parametric model, and then apply spatial interpolation to the parameters. The work~\cite{carta2009} lists approximately 200 papers written between 1940 and 2008 which focus on parametric models for time series of wind speed measurements. Due to high spatial variability, interpolating wind over shorter time intervals is considered a harder problem. However, there are reasons why this can be useful, such as predicting the propagation of forest fires and pollutants, modelling the movement of flying animals and insects~\cite{Luo2008}, or setting up initial conditions for weather simulations.\\

Interpolating both wind speed and direction from a sparse set of wind velocity measurements is commonly referred to as wind field reconstruction. In modern approaches to field reconstruction such as \cite{Callaham2019, jin2020, Erichson2020}, high resolution numerical simulations are used to train machine learning models to astonishing accuracy. However, this type of high resolution training data is sometimes not accessible, or might take a considerable time and effort to simulate. In such cases, simpler interpolation models can arguably be a valid alternative. The definition of an interpolation model varies depending on the context. The definition used in our work draws from \cite{hengl2018}, and is synonymous with regression (see Section~\ref{sec: Spatial interpolation models}).\\

A number of studies have been made comparing different interpolation models in various applications such as  temperatures~\cite{appelhans2015}, snow depth~\cite{erxleben2002}, mineral concentrations~\cite{hengl2018} and wind speed~\cite{cellura2008, carta2009, jung2015}. Traditional spatial interpolation models include nearest neighbours, inverse distance weighting,  Kriging, linear regression, polynomial spline methods and various combinations of them~\cite{JinBook2008}. In recent years, machine learning methods such as random forests and neural networks have proven worthy adversaries of traditional interpolation models, see for example~\cite{appelhans2015} and \cite{hengl2018}. These types of models trade interpretability for power and flexibility and allow for the inclusion of features such as terrain elevation, slope, concavity, roughness~\cite{appelhans2015}, without much additional work.\\

In this paper, we compare methods for near surface wind field reconstruction. A machine learning method known as random Fourier features is compared to a selection of popular and successful interpolation techniques. The model fits a Fourier series $\beta(x) = \sum_{k} \beta_k e^{i\omega_k x}$ to the data. Instead of traditional greedy optimisation methods such as stochastic gradient descent, our random Fourier features model explores the frequency domain using an adaptive Metropolis algorithm inspired by \cite{kammonen2020}. In each step, the Fourier coefficients $\beta_k$ are optimised with respect to a loss function. The work \cite{tempone99} proposes a technique for interpolation of incompressible flow by incorporating zero divergence as a constraint in the optimisation. We propose an alternative approach where the incompressibility takes the form of a regularisation term. In Section~\ref{sec: Problem formulation}, a mathematical formulation of the wind field reconstruction problem is formulated. The data is also presented, along with error estimates used to evaluate the models. Section~\ref{sec: Models} contains a short introduction to the interpolation models. The results are presented in section~\ref{sec: Results} and finally, discussed in Section~\ref{sec: Discussion}.
\section{Problem formulation}\label{sec: Problem formulation}
Let $\Omega\subset \R^2$ denote a geographic region. For the purpose of this report, $\Omega$ is the set of points contained within the Swedish borders. We define the horizontal wind field $\pmb{u}\colon \Omega\times[0,\infty)\to \R^2$ that maps every point in space $\pmb{x}=(x,y)\in \Omega$ and time $t>0$ to a velocity vector $\pmb{u}(\pmb{x},t)=(u(\pmb{x},t), v(\pmb{x}, t))\in \R^2$. Typically, air flow is assumed to satisfy the incompressible Navier-Stokes differential equations. Given an initial state at the time $t=0$ and a set of boundary conditions, it is possible to simulate and forecast the wind using these equations. In this report, the focus is shifted from forecasting to interpolation. In practice, this is a hard problem because wind measurements are only accessible in a sparse set of points in $\Omega$. We call this problem \textit{wind field reconstruction}. The above notation as well as the notation presented in sections~\ref{sec: Data} and \ref{sec: Spatial interpolation models} will be used extensively throughout our work.
\subsection{Data}\label{sec: Data}
The data was obtained from the Swedish Meterological and Hydrological Institute (SMHI). It contains a set of hourly wind velocity observations during the entirety of 2018, from $N=171$ weather stations scattered across Sweden  as shown in Figure~\ref{fig: station scatter}. Each measurement is collected 10 meters above ground. The positions of the stations are given by the altitude as well as the latitude and longitude. Each measurement is a 10-minute average and consists of two components: The wind speed, which is measured in meters per second, and the angle of the horizontal component of the wind vector, measured counter-clockwise relative to north.  Both measurements are rounded to varying degrees of precision depending on the station~\cite{SMHI}. The stations are not active at all times. In fact, there are occasional hours with as few as one station reporting measurements. The velocities are highly correlated over time, as demonstrated in Figure~\ref{fig: u-autocorr}. The data was processed before training. First, the latitude-longitude pairs were transformed to cartesian coordinates $\pmb{x}=(x,y)$ where $x$ is the eastward-measured distance and $y$ is the northward measured distance, as shown in Figure~\ref{fig: station scatter}. This was done using the \texttt{SWEREF~99~TM}\footnote{More information about the \texttt{SWEREF~99~TM} map projection can be found here:\\ \href{https://www.lantmateriet.se/en/maps-and-geographic-information/gps-geodesi-och-swepos/Referenssystem/Tvadimensionella-system/SWEREF-99-projektioner}{\textcolor{blue}{https://www.lantmateriet.se/en/maps-and-geographic-information/gps-geodesi-och-swepos/Referenssystem/Tvadimensionella-system/SWEREF-99-projektioner}.}} map projection. Secondly, the wind measurements were transformed from polar to cartesian coordinates $\pmb{u}=(u,v)$, where $u$ corresponds to the velocity component along the $x-$axis and $v$ corresponds to the velocity component along the $y-$axis. Lastly, the wind measurements from September 2018 were removed because of unusually high wind speeds. 

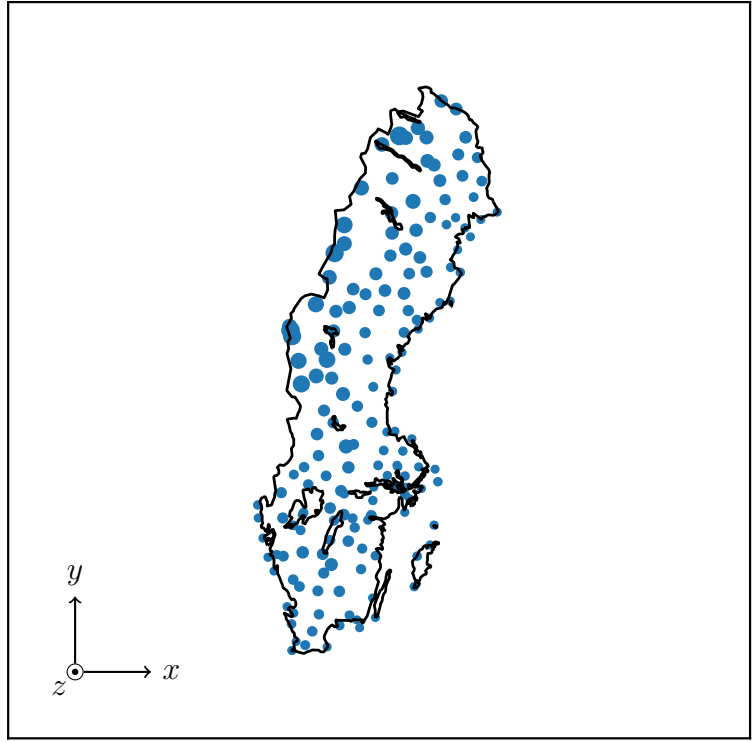
\begin{figure}
    \centering
    \begin{tikzpicture}
    \node at (0,0) {\includegraphics[width=0.7\linewidth]{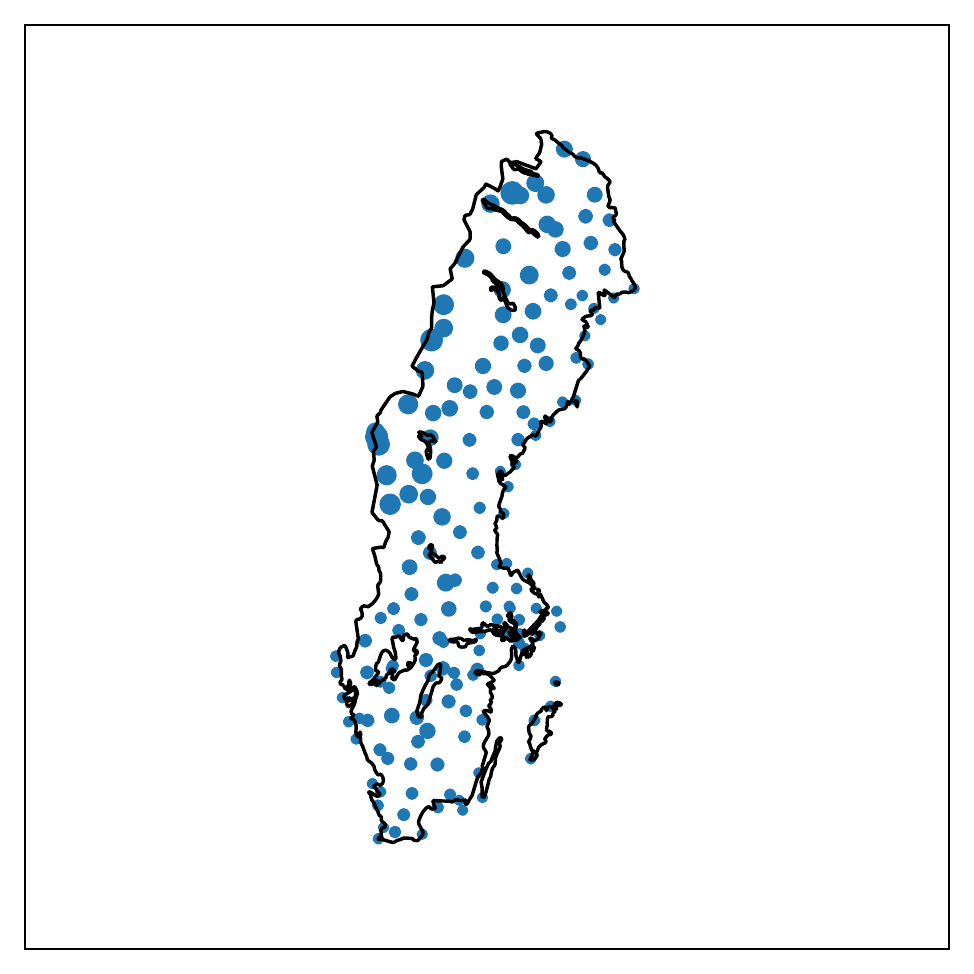}};
    \draw[thick, ->](-4,-4)--(-4,-3) node[above] {$y$};
    \draw[thick, ->](-4,-4)--(-3,-4) node[right] {$x$};
    \filldraw[white] (-4,-4) circle (3pt);
    \draw (-4,-4) circle (3pt);
    \filldraw (-4,-4) circle (1pt);
    \node[xshift=-6pt, yshift =-6pt] at (-4,-4) {$z$};
    \end{tikzpicture}
    \captionsetup{width=0.7\linewidth}
    \caption{The \texttt{SWEREF 99 TM} orthographic projection of Sweden. The weather stations are marked in blue and scaled according to their altitudes $z$. The coordinate system is drawn in the bottom left.}
    \label{fig: station scatter}
\end{figure}
\subsection{Spatial interpolation models}\label{sec: Spatial interpolation models}
For the purpose of this report, the definition of interpolation models is restricted to approximations of functions $\pmb u\colon \Omega\to \R^2$. The range is two dimensional since we are modelling horizontal wind velocities, and the region $\Omega$ is Sweden. In general, interpolation models are also used for other properties like temperature, pressure, population density etc. Nevertheless, we define a \textit{spatial interpolation model} as a map $f$ from a set of velocity measurements $\mathcal{D}=\{(\pmb x_n, \pmb u_n)\colon \pmb x_n \in \Omega,\; n=1,2,\dots,N\}$ to a vector field $f_{\mathcal{D}}\colon \Omega\to \R^2$. The process of evaluating a model $f$ on $\mathcal{D}$ is called \textit{training}. Usually, the training is done by minimising a loss function. The above notation will be used throughout our report. Note that there is an important distinction between $f$ and $f_{\mathcal{D}}$. The function $f$ represents a model, which is trained on data $\mathcal{D}$ and produces a vector field $f_{\mathcal{D}}$ which approximates $\pmb u$. Any time a symbol is indexed by the letter $\mathcal{D}$ or variations of it such as $\mathcal{D}_t$, it symbolises an interpolation model trained on that particular data set. A parametrised set of interpolation models $\{f^\theta\colon \theta\in\Theta\}$ of interpolation models $f^\theta$ is called an \textit{interpolation family}, and the parameters $\theta$ are called \textit{hyper parameters}. For example, the $K$-nearest neighbours models can be viewed as an interpolation family with hyper parameter $K$. The process of finding the best model $f^\theta$ in an interpolation family with respect to some quality of fit is called \textit{hyper optimisation}. In the next section, we present a quality of fit used for hyper optimisation and comparisons between different interpolation models. 
%
\subsection{Quality of fit}\label{sec: Quality of fit}
Given an interpolation model $f$ and observations $\{\mathcal{D}_t\}_{t\in T}$ at different times $t$ in a time span $T$, the quality of fit $\mathcal{Q}(f)$ is defined as the expected square loss with respect to the distribution $\rho(\mathrm{d}\pmb x,\mathrm{d}t)$ of the data in space and time $\Omega\times T$:
\[
    \mathcal{Q}(f) = \mathrm{E}\left[||f_{\mathcal{D}_t}(\pmb x)-\pmb u(\pmb x, t)||^2\right] = \int_{T}\int_{\Omega}||f_{\mathcal{D}_t}(\pmb x)-\pmb u(\pmb x,t )||^2\rho(\mathrm{d}\pmb x,\mathrm{d}t).
\]
We only have access to a limited sparse set of data, so calculating the exact value of $\mathcal{Q}$ is not possible. Even if the number of measurements was sufficient, model training is often cumbersome. Instead, we use a Monte-Carlo sampling average in time and a cross-validation scheme in space. Let $\{t_k\}_{k\in\mathcal{K}}$ be a set of independent time samples from $T$, indexed by a set $\mathcal{K}$, and take the measurements $\mathcal{D}_{k}=\{(\pmb x_n, \pmb u_{kn})\colon n=1,2,\dots, N_k\}$ to be the set of observations made at the time $t_k$. As stated in Section~\ref{sec: Data}, the weather stations are not always active, so the number of measurements $N_k$ at the time $t_k$ varies. Define a random partition of the weather stations into $M$ disjoint sets of equal size $\mathcal{D}_{k}^m, m=1,2,\dots,M$, and denote $\mathcal{D}_{k}^{-m} = \mathcal{D}_k\setminus \mathcal{D}_k^{m}$. The conditional expectations $\mathcal{Q}_t(f) = \mathrm{E}\left[||f_{\mathcal{D}_t}-\pmb u(\pmb x, t)||^2\mid t\right]$ of the quality of fit given the times $t_k$ are approximated using cross-validation:
\[
    \mathcal{Q}_{t_k}(f) \approx \widetilde{\mathcal{Q}}_k(f) \coloneqq \frac{1}{|\mathcal{D}_{k}|}\sum_{m=1}^M\sum_{(\pmb x',\pmb u')\in\mathcal{D}_k^m}|| f_{\mathcal{D}_k^{-m}}(\pmb x')-\pmb u'||^2.
\] 
We found that 5-fold cross-validation ($M=5$) struck a good balance between computation time and accuracy. Averaging $\widetilde{\mathcal{Q}}_k(f)$ over the samples in $\mathcal{K}$ yields a Monte Carlo sample estimate of $\mathcal{Q}(f)$:
\[
    \mathcal{Q}(f) \approx \widetilde{\mathcal{Q}}(f)\coloneqq \frac{1}{|\mathcal{K}|}\sum_{k\in \mathcal{K}}\widetilde{\mathcal{Q}}_k(f),
\]
also called the \textit{unexplained variance}. An alternative measure $\mathcal{E}$ is obtained from normalising $\mathcal{Q}(f)$ by the expected square wind speed $\mathrm{E}_{t,\pmb x}[||\pmb u(\pmb x, t)||^2]$, denoted $\mathcal{Q}(0)$:
\begin{equation}
    \label{eq: unexplained variance}
    \mathcal{E}(f) = \frac{\mathcal{Q}(f)}{\mathcal{Q}(0)} \approx \frac{\widetilde{\mathcal{Q}}(f)}{\widetilde{\mathcal{Q}}(0)}=:\widetilde{\mathcal{E}}(f).
\end{equation}
This measurement is referred to as the \textit{fraction of unexplained variance}. If the wind field $\pmb u$ has zero mean, the number $1-\mathcal{E}(f)$ is referred to as $\mathcal{R}^2$, or the \textit{coefficient of determination}. We estimate the variance of $\widetilde{\mathcal{Q}}$ as
\begin{equation}
    \mathrm{Var}(\widetilde{\mathcal{Q}}) =\frac{\mathrm{Var}(\mathcal{Q}_t)}{|\mathcal{K}|},\qquad \mathrm{Var}(\mathcal{Q}_t) \approx \frac{1}{|\mathcal{K}|}\sum_{k\in\mathcal{K}}(\widetilde{\mathcal{Q}}_k-\widetilde{\mathcal{Q}})^2.
    \label{eq: clt std}
\end{equation}
Provided that the data samples are independent and that the cross-validation estimates $\widetilde{\mathcal{Q}}_k$ of $\mathcal{Q}_{t_k}$ are exact and the sample size $|\mathcal{K}|$ is sufficiently large, the central limit theorem guarantees that the quality of fit follows a normal distribution. As a direct consequence, two standard deviations of $\widetilde{\mathcal{Q}}$ constitute an approximate 5th percentile confidence bound for $\mathcal{Q}$. Note that the confidence bound for $\mathcal{Q}$ is not sufficient when comparing models, since the errors can be highly correlated. Given two models $f,g$ we instead define the difference $\Delta \mathcal{Q}(f,g) := \mathcal{Q}(f)-\mathcal{Q}(g)$ and estimate it as $\Delta\widetilde{\mathcal{Q}}(f,g) := \widetilde{\mathcal{Q}}(f)-\widetilde{\mathcal{Q}}(g)$. The variance of $\Delta \widetilde{\mathcal{Q}}(f,g)$ is estimated similarly to \eqref{eq: clt std}:
\begin{multline}
    \mathrm{Var}(\Delta\widetilde{\mathcal{Q}}(f,g)) =\frac{\mathrm{Var}(\Delta\mathcal{Q}_t(f,g))}{|\mathcal{K}|},\quad \text{where}
    \\\mathrm{Var}(\Delta \mathcal{Q}_t(f,g)) \approx \frac{1}{|\mathcal{K}|}\sum_{k\in\mathcal{K}}(\Delta\widetilde{\mathcal{Q}}_k(f,g)-\Delta\widetilde{\mathcal{Q}}(f,g))^2.
    \label{eq: diff std}
\end{multline}
The equations \eqref{eq: clt std} and \eqref{eq: diff std} were used in Section~\ref{sec: Results} to determine the required sample size $|\mathcal{K}|$ for the 5th percentile confidence intervals to determine $\mathcal{Q}(f)$ and $\Delta\mathcal{Q}(f,g)$ with a precision of 1\% of the mean square wind speed. This was done with the 2018 wind data set for all the tested models $f$ and $g$ as the random Fourier features. Due to high autocorrelation and seasonality in the wind as demonstrated in Figure~\ref{fig: u-autocorr}, these confidence intervals do not necessarily generalise to larger time intervals $T$ than 2018. The quality of fit was used for both hyperparameter optimisation and validation of the models, but with different samples $\mathcal{K}$ to avoid overfitting.
\section{Models}\label{sec: Models}
\subsection{Fourier models}
In this section we introduce two Fourier series based models. At the end of the section, we arrive at the random Fourier features model, which is the main focus of our report. In Section~\ref{sec: Benchmarking}, we lay out some of the conventionally used interpolation models which are used for bench-marking. 
\subsubsection{Fourier series}\label{sec: Fourier series}
The Fourier series based model takes the form
\[
    \beta(\pmb x) = \sum_{k=1}^{K} \beta_k e^{i\omega_k\cdot\pmb x},\qquad \omega_k \in \R^2, \; \beta_k \in \C^2,\quad k = 1, 2,\dots, N,
\]
where $K$ is the number terms and $\omega_k\cdot \pmb x$ denotes the scalar product between $\omega_k$ and $\pmb x$. The Fourier series generalises to arbitrary dimensions of the input $\pmb x$, but in this report we chose $\pmb x = (x,y)$ to be just the horizontal coordinates, hence why $\omega_k$ is two-dimensional. In the Fourier series based model, $\pmb\omega=(\omega_1,\omega_2,\dots,\omega_K)$ is held fixed, and the parameters $\pmb{\beta} = (\beta_1,\beta_2, \dots , \beta_K)$ are estimated by optimising with respect to the expectation of a loss function $\ell\colon \R^2\times \R^2\to \R^+$:
\[
    \min_{\pmb{\beta}}\quad \mathrm{E}_{\tilde{\rho}}\left[\ell(\beta(\pmb x),\pmb u)\right],
\]
where the expectation is taken over the joint density $\tilde\rho$ of the data $\pmb x,\pmb u$. For our application, $\tilde{\rho}$ is unknown. Therefore, the expected loss is replaced by a Monte-Carlo sample estimate of the loss, called the \textit{empirical loss}. Given a data set $\mathcal{D}=\{(\pmb x_n, \pmb u_n)\}_{n=1}^N$, the empirical loss is expressed as
\[
    \frac{1}{N}\sum_{n=1}^N \ell(\beta(\pmb x_n), \pmb u_n) \approx \mathrm{E}_{\tilde{\rho}}\left[\ell(\beta(\pmb x), \pmb u)\right].
\]
For this work, we chose the following loss function:
\begin{equation}
    \label{eq: loss}
    \ell(\beta(\pmb x), \pmb u) = ||\beta(\pmb x)-\pmb u||^2 + \lambda ||\beta||_{S(2,2)}^2 + \eta ||\nabla \cdot \beta||_{L^2}^2.
\end{equation}
The expression $||\beta ||_{S(2,2)}^2$ denotes the second order squared Sobolev-norm of $\beta$. By orthogonality of the Fourier features, the Sobolev norm can be simplified to
\[
    ||\beta||_{S(2,2)}^2 = \sum_{k=1}^K \left(\gamma ^2||\omega_{k}||^4 + \gamma||\omega_k||^2 + 1\right)||\beta_k||^2.
\]
The hyper parameter $\gamma$ was added to allow for more flexibility in the penalty, and is equivalent to rescaling the input variable before applying the Sobolev norm. The expression $||\nabla \cdot \beta||_{L^2}^2$ is the squared $L^2$-norm of the divergence of $\beta$, which can be rewritten as
\[
||\nabla\cdot \beta||_{L^2}^2 = \sum_{k=1}^K  |\omega_k\cdot \beta_k|^2.
\]
The intended effect of the Sobolev norm is to dampen high frequencies, and the divergence penalty is supposed to simulate incompressible flow. The empirical loss using the loss function as defined in \eqref{eq: loss} is
\begin{equation}
    \frac{1}{N}\sum_{n=1}||\beta(\pmb x_n)-\pmb u_n||^2 + \lambda ||\beta||_{S(2,2)}^2 + \eta||\nabla \cdot \beta||_{L^2}^2.
    \label{eq: empirical loss}
\end{equation}
Here, $\lambda$ and $\eta$ are hyper parameters. In order to determine a suitable choice for $\pmb{\omega}$, assume the standard regression setting $u=f(x)+\epsilon$, where $\epsilon$ is zero mean and independent of $\pmb x$. Since the spatial region $\Omega$ is bounded, $f$ can be extended periodically over $\R^2$. That is, the relation $f(x+m\tau_x, y+n\tau_y)=f(x,y)$ is imposed on $f$ for all $\pmb x=(x,y)\in\R^2$, whole numbers $m,n$ and some two-dimensional period $\pmb\tau = (\tau_x,\tau_y)$. This means that $f$ can be written as a Fourier series. Thus, $f$ can in theory be approximated arbitrarily well by  $\beta$ as the number of terms $K$ tends to infinity. However, since there is only a limited amount of data and $\beta$ contains a limited number of terms, the choice of $\pmb\omega$ determines how well $\beta$ can approximate $f$. In the Fourier series based model, we settle on choosing $\pmb\omega$ as a square grid with side length $2M+1$, centered at the origin in the frequency domain:
\[
    \pmb\omega = \{(\pi\tfrac{m}{\tau_x}, \pi\tfrac{n}{\tau_y})\colon -M \leq m\leq M, \quad -M\leq n\leq M\}.
\]
As such, the Fourier series based model is a spatial interpolation family with the hyper parameters $\lambda, \eta, M$ and $\pmb \tau$. We went with $M=10$, and $\lambda,\eta,\pmb \tau$ were chosen as explained in the upcoming section.

\begin{remark} Minimising the expression in \eqref{eq: empirical loss} amounts to solving a system of linear equations with respect to the elements $\pmb \beta$. First define the $N\times K$ matrix $\pmb S$ such that $\pmb{S}_{n,k}=e^{i \omega_k\cdot \pmb x_n}$ for $k=1,\dots,K$ and $n=1,\dots,N$. Secondly, define a matrix $\pmb U_{n,i}=u_{n,i}$ for $i=1,2$ and $n=1,2,\dots,N$. Lastly, define three diagonal $K\times K$ matrices $\pmb \Lambda, \pmb D_1$ and $\pmb D_2$ such that $\pmb\Lambda_{kk} = \lambda(\gamma^2||\omega_k||^4+\gamma||\omega_{k}||^2 +1)$ and $\pmb D_{kki} = \sqrt{\eta}\omega_{ki}$ for $i=1,2$ and $k=1,2,\dots, K$. Denoting the row vectors of $\pmb\beta$ and $\pmb{U}$ by $\pmb\beta_i$ and $\pmb{U}_i$ respectively, the empirical loss takes the form
\[
    (\pmb D_1 \pmb \beta_1 + \pmb D_2\pmb \beta_2)^\dag (\pmb D_1 \pmb \beta_1 + \pmb D_2\pmb\beta_2)+\sum_{i=1,2} \tfrac{1}{N}(\pmb{S}\pmb{\beta}_i - \pmb{U}_i)^\dag(\pmb{S}\pmb{\beta}_i - \pmb{U}_i) + \pmb{\beta}_i^\dag \pmb{\Lambda} \pmb{\beta}_i,
\]
where $A^\dag$ denotes the Hermitian conjugate of the matrix $A$. Differentiating with respect to the real and imaginary parts of $\pmb \beta$, setting to zero and adding the equations together results in the system
\begin{align*}
    (\tfrac{1}{N}\pmb{S}^\dag\pmb{S}+\pmb\Lambda +\pmb D_1^\dag \pmb D_1)\pmb\beta_1 + \pmb D_1^\dag \pmb D_2 \pmb \beta_2&= \tfrac{1}{N}\pmb{S}^\dag\pmb{U}_1,\\
    (\tfrac{1}{N}\pmb{S}^\dag \pmb{S}+\pmb\Lambda + \pmb D_2^\dag \pmb D_2)\pmb\beta_2 + \pmb D_2^\dag \pmb D_1 \pmb  \beta_1 &= \tfrac{1}{N}\pmb{S}^\dag\pmb{U}_2.
\end{align*}
This is a linear system of equations which can be solved for example by singular value decomposition. The most noticeable difference from Tikhonov regression is the interaction between $\pmb \beta_1$ and $\pmb \beta_2$ due to the divergence penalty.
\end{remark}
\subsubsection{Random Fourier features}\label{sec: Fourier features}
The random Fourier features model and the Fourier series based approach start off similarly:
\[
    \beta(x) = \sum_{k=1}^{K} \beta_k e^{i\omega_k\cdot \pmb x},\qquad \omega_k \in \R^2, \; \beta_k \in \C^2,\quad k = 1, 2,\dots, N,
\]
The difference is that instead of just optimising with respect to $\pmb{\beta} = (\beta_1,\beta_2, \dots , \beta_K)$, random Fourier features aims at solving the harder problem
\[
    \min_{\pmb{\beta}, \pmb{\omega}}\quad \mathrm{E}\left[\ell(\beta(\pmb x),\pmb u)\right].
\]
That is, we also want to optimise with respect to $\pmb\omega$. In this way, the random Fourier features is similar to a fully connected neural network with one hidden layer and the activation function $x\mapsto e^{ix}$. The optimisation is traditionally done using some greedy method such as stochastic gradient descent, but in random Fourier features the expected loss is approximated by viewing ${\pmb\omega}$ as a random variable:
\[  
    \min_{\pmb{\beta}, \pmb{\omega}} \mathrm{E}\left[\ell(\beta(\pmb x), \pmb u)\right] \leq \mathrm{E}\left[\min_{\pmb{\beta}}\mathrm{E}\left[\ell(\beta(\pmb x),\pmb u)\mid \pmb \omega\right]\right].
\]
We use the loss function defined in \eqref{eq: loss}. Furthermore, we assume same the standard regression setting  $\pmb u=f(\pmb x)+\epsilon$ as presented in the previous section, where $\epsilon\in\R^2$ is zero mean and independent of $\pmb x$ and $f$ a periodic function. Lastly, the elements of $\pmb\omega$ are assumed to be independent and identically distributed according to the density $\rho$. In the Appendix we derive an upper bound
\[  
    \mathrm{E}\left[\min_{\pmb{\beta}}\mathrm{E}\left[\ell(\beta(\pmb x),\pmb u)\mid \pmb\omega\right]\right] \leq \frac{1+\lambda \overline{C}}{(2\pi)^2 K}\sqrt{\mathrm{E}\left[\frac{||\hat f(\omega)||^4}{\rho(\omega)^4}\right]} -\frac{1}{K}E[||f(\pmb x)||^2]+ \mathrm{E}[||\epsilon||^2]
\]
using the above assumptions as well as assuming that the distribution $\rho$ is discrete and has bounded moments up to a degree determined by the order of the derivatives used in the regularisation. Furthermore, we show that this upper bound is minimised by choosing $\rho(\omega)\propto||\hat f(\omega)||$, where $\hat f(\omega)$ are the Fourier coefficients for the Fourier series expansion of $f$. Since the function $||\hat f(\omega)||$ is not known a priori, this distribution has to be approximated somehow. The authors of \cite{kammonen2020} present an adaptive Metropolis algorithm for sampling from $\rho$ in a related setting. The main differences are that $f$ is $L^2$ integrable on the entirety of $\R^2$ and has a one-dimensional range. Drawing from the work in \cite{kammonen2020}, we devise an adaptive Metropolis algorithm for sampling the weights $\omega$ from $||\hat f ||$ as seen in Algorithm~\ref{alg: RFF}.
\begin{algorithm}
\SetAlgoLined
\SetKwData{Left}{left}
\SetKwData{This}{this}
\SetKwData{Up}{up}
\SetKwFunction{Union}{Union}
\SetKwFunction{FindCompress}{FindCompress}
\SetKwInOut{Input}{input}
\SetKwInOut{Output}{output}
\Input{Rescaled data $\{x_n, y_n\}_{n=1}^N\subset [0,1]^2\times  \R^2$}
\Output{Random features $x\mapsto \sum_{k=1}^K \beta_k e^{i\omega_k\cdot x}$}
 $K\leftarrow$ Choose the number of frequencies\;
 $B\leftarrow$ Choose the number of steps\;
 $\sigma\leftarrow$ Choose a variance for the proposal kernel\;
 $\gamma\leftarrow$ Choose the exponent for the transition probability\;
 $\lambda\leftarrow$ Choose the Sobolev regularisation parameter\;
 $\eta\leftarrow$ Choose the divergence regularisation parameter\;
 $\pmb\omega\leftarrow$ the zero vector in $\R^{2K}$\;
 $\pmb\beta\leftarrow$ minimiser of the empirical~\eqref{eq: empirical loss} loss given $\pmb{\omega}$\;
 \For{$b\leftarrow 1,\dots, B$}{
    $r_{\mathcal{N}}\leftarrow$ standard normal random vector in $\R^{2K}$\;
    $r\leftarrow$ round the elements of $\sigma r_{\mathcal{N}}$ to the nearest integer\; 
    $\pmb{\omega}'\leftarrow\pmb{\omega}+r$\;
    $\pmb\beta'\leftarrow$ minimiser of the empirical loss~\eqref{eq: empirical loss} given $\pmb{\omega}'$\;
    \For{$k\leftarrow 1,\dots , K$}{
        $\alpha \leftarrow$ sample from uniform distribution on $[0,1]$\;
        \If{$||\beta_k'||^\gamma/||\beta_k||^\gamma > \alpha$}{
            $\omega_k\leftarrow\omega_k'$\;
            $\beta_k\leftarrow\beta_k'$\;
        }
    }
 }
 $\pmb\beta\leftarrow$ minimiser of the empirical loss~\eqref{eq: empirical loss} given $\pmb\omega$\;
 $x\mapsto \sum_{k=1}^K \beta_k e^{i\omega_k\cdot x}$
 \caption{Discrete random Fourier features with Metropolis sampling}
 \label{alg: RFF}
\end{algorithm}
The random Fourier features algorithm can be parametrised as $\{\beta_\theta\colon \theta\in\Theta\}$ where the hyper parameter $\theta$ consists of the number of fourier frequencies $K$, the number of steps $B$, the periodicity $\tau$, the regularisation parameters $\sigma$ and $\gamma$, as well as the parameters $\lambda$ and $\eta$ in the Metropolis sampler. Similarly to the work~\cite{tempone99}, $\tau$ was chosen between two and three times the size of the region of interest, $\tau = (4\cdot 10^6,4\cdot 10^6)$. The number of frequencies was fixed to $400$ and the number of steps $B$ was fixed to $500$. Given a $d$-dimensional feature space, \cite{kammonen2020} shows that optimal choice for $\gamma$ as $K,N\to\infty$ with a fixed computational work is $3d-2$. Furthermore, the classical result from~\cite{oroberts2001} shows that the optimal variance of the proposal kernel for a general random walk Metropolis-Hastings algorithm is $\sigma \approx \frac{2.4^2}{d}$.\\\\
Since theoretical results listed above hinge on limit cases, they are not guaranteed to hold for the sparse wind measurements. Therefore, a partial grid search minimisation of the unexplained variance $\mathcal{E}(\beta_\theta)$ from \eqref{eq: unexplained variance} was carried out in a neighbourhood of the theoretically optimal values for $\gamma$ and $\sigma$, keeping $\tau, \lambda$ and $\eta$ fixed. Given estimations of the optimal $\gamma$ and $\sigma$, a similar grid search was then carried out on the parameters $\lambda$ and $\eta$. The number of steps was adjusted to insure convergence of the Markov chain, while also keeping the computation time at a minimum. Note that the random Fourier features algorithm is not guaranteed to find the optimal frequencies. There is ongoing research in this area, and an iterative greedy method which finds the optimal frequencies is currently in the works \cite{unpublished_luis}.

\subsection{Benchmarking models}\label{sec: Benchmarking}
The following interpolation models were used for benchmarking: Nearest neighbours, inverse distance weighting (IDW for short), Kriging, random forest, neural networks and a Fourier series based model introduced in the previous section. This section will serve as a short introduction to each of the methods as well as motivation as to why they are relevant. We use the same notation $\mathcal{D}=\{(\pmb x_n, \pmb u_n)\in \Omega\times\R^2\colon n=1,2,\dots,N\}$ for the measurements as in Section~\ref{sec: Spatial interpolation models}.
\subsubsection{Inverse distance weighting}\label{sec: IDW}
Inverse distance weighting methods $\{f^p\colon p \geq 0\}$ is an interpolation family of methods which are evaluated at a point $\pmb x$ by taking a weighted average of the velocity measurements in $\mathcal{D}$. Specifically, the weights $\alpha(\pmb x_n,\pmb x)$ for a model $f^p$ from the IDW family are proportional to $1/d(\pmb x_n,\pmb x)^p$ where  $d\colon \Omega\times \Omega \to [0,\infty)$ is a distance on $\Omega$. That is,
\[
    f_{\mathcal{D}}^p(\pmb x) = \frac{\sum_{n=1}^N \alpha(\pmb x_n, \pmb x)\pmb u_n}{\sum_{n=1}^N\alpha(\pmb x_n, \pmb x)},\quad \text{where}\quad\alpha(\pmb x_n,\pmb x)=\frac{1}{d(\pmb x_n, \pmb x)^p}.
\]
The singularities at $\pmb x=\pmb x_n, n=1,2,\dots,N$ are removed by setting $f_{\mathcal{D}}^p(\pmb x_n) = \pmb u_n$. The hyper parameter $p$ adjusts the amount of influence each data point has over its immediate surroundings. Letting $p=0$ will result in all points weighing equally everywhere, i.e. taking the arithmetic mean of the data. Letting $p\to\infty$ will result in the nearest neighbor method. Usually, $p$ is chosen somewhere inbetween. A common shortcoming of IDW is that the interpolated values are bounded by the maximum and minimum values of the data and therefore IDW fails to predict unobserved extreme points. The main benefits are interpretability and relatively short training time. In this report, two values of $p$ were tested, namely $p=2$ and $p=\infty$ (i.e. nearest neighbors). Furthermore, we used the horizontal distance between points, ignoring altitudes.
\subsubsection{Kriging}\label{sec: Kriging}
Kriging is a statistical approach to spatial interpolation. The true velocity $\pmb u$ is assumed to satisfy the equality
\[
    \pmb u(\pmb x) = \mu(\pmb x) + \epsilon_{\pmb x}, 
\]
where $\mu\colon \Omega\to \R^2$ is a deterministic function and $\epsilon_{\pmb x}$ is a zero mean stochastic process over $x$ with a specific covariance structure given by the \textit{covariogram} $C$:
\[
    \mathrm{Cov}(\epsilon_{\pmb x}, \epsilon_{\pmb x'}) = \mathrm{E}[\epsilon_{\pmb x}\epsilon_{\pmb x'}] = C(|\pmb x-\pmb x'|).
\]
Kriging works by first using a deterministic model to estimate the mean $\mu$ using the data, resulting in a trained model $\mu_{\mathcal{D}}$. For each $(\pmb x_n,\pmb u_n)\in\mathcal{D}$, the residuals $\epsilon_{\pmb x_n}$ are estimated as $\epsilon_{\pmb x_n}\approx \epsilon_n = \pmb u_n-\mu_{\mathcal{D}}(\pmb x_n)$. Next, the covariogram is fitted using the residuals $\epsilon_n$. The last part of the fit is to, given an input $\pmb x$, find a linear combination $\sum_n\omega_n\epsilon_n  = \omega^T \epsilon$ such that the weights $\omega$ minimise the expected square error $\mathrm{E}[(\epsilon_{\pmb x} - \sum_n\omega_n\epsilon_{\pmb x_n})^2]$. This is equivalent to solving the linear system of equations 
\[
    \sum_{m=1}^NC(||\pmb x_n-\pmb x_m|)\omega_m = C(|\pmb x_n - \pmb x|), \quad n=1,2,\dots, N
\]
Thus, defining the $N\times N$ matrix $\pmb{C}_{mn}=C(|\pmb x_m-\pmb x_n|)$ and the $N\times 1$ vector $c_n(\pmb x)=C(|\pmb x_n-\pmb x|)$ for $n,m = 1,2,\dots,N$, the final estimation $f_{\mathcal{D}}$ of $u$ takes the form
\[
    f_{\mathcal{D}}(\pmb x) = \mu_{\mathcal{D}}(\pmb x) + \epsilon^T \pmb C^{-1}c(\pmb x).
\]
By also modelling the so-called \textit{variogram} $E[(\epsilon_{\pmb x}-\epsilon_{\pmb x'})^2]$, it is possible to calculate the variance of the residuals $\epsilon_{\pmb x}$ and thereby obtain an estimate of the uncertainty at each point $\pmb x$. Furthermore, Kriging can be combined with virtually any other unbiased deterministic interpolation model by interpreting it as the mean $\mu(\pmb x)$. However, when the statistical assumptions do not hold, Kriging can perform poorly. Machine learning methods such as random forests have been successful in beating Kriging for various spatial interpolation tasks, see for example~\cite{appelhans2015, hengl2018}. The authors in~\cite{hengl2018} argue that even though Kriging might be redundant in terms of accuracy, it remains valueable tool for understanding data, exactly because of its statistical properties and interpretability. We used a version of Kriging called \textit{Universal Kriging}, which differs from Kriging in that a linear regression approximation of the mean $\mu(\pmb x)$ and the residuals $\epsilon_{\pmb x}$ are fitted to the data simultaneously, resulting in a joint system of equations for all the weights. The python package \texttt{pykrige} was used to implement Universal Kriging with a linear variogram.
\subsubsection{Random forest}\label{sec: random forest}
Random forests are constructed by averaging an ensemble of regression trees. Each tree is trained on a random sample of the data $\mathcal{D}$ and each split in the tree is chosen by randomly selecting one feature out of the input features~\cite{breiman2001}. Random forests have been used sucessfully in spatial interpolation problems, for example to predict temperatures on and around Kilimanjaro, Tanzania \cite{appelhans2015} as well as mineral concentrations \cite{hengl2018}. The main drawback of random forests is lack of interpretability. In this report we used a random forest with $200$ trees with mean square loss for splitting, and unlimited tree depth. The forest was implemented in \texttt{scikit-learn}. Furthermore, the Random forest was trained on a polynomial feature map $\phi$ of the horizontal coordinates $\pmb x= (x,y)$ and the altitude $z$. Specifically, $\phi$ maps the coordinates to all polynomial features $1^{p_1}x^{p_2}y^{p_3}z^{p_4}$ with a total order $p_1+p_2+p_3+p_4$ of $\leq 3$. That is, the trained model $f_{\mathcal{D}}$ is a composition $h\circ \phi$ of the feature map $\phi$ and the random forest $h_{\mathcal{D}}\colon \R^{20}\to \R^2$.
\subsubsection{Feedforward neural network}
Feedforward neural networks have been used extensively in different areas of applied mathematics. In this report, only a specific family of feedforward neural networks were considered. Namely, the networks are characterised by $3$ fully connected hidden layers, a constant number of $n$ nodes in each layer and the ReLU activation function. The input layer consists of the three spatial coordinates $x,y$ and $z$. The weights were optimised using the Adam algorithm~\cite{kingma2017adam}, with respect to the $L_2$-regularised loss. The network was implemented in \texttt{TensorFlow}, and hyper parameter optimisation of the number of nodes and regularisation parameter was done with a grid search on the unexplained variance. 
\subsubsection{Weighted linear combination of model}
Given a set of $n$ interpolation models $\{f^1, f^2, \dots f^n\}$ and a data set $\mathcal{D}_t$, we can improve on the individual models by forming a weighted average
\[
    f_{\mathcal{D}_t} = \alpha_1 f_{\mathcal{D}_t}^1 + \alpha_2 f_{\mathcal{D}_t}^2 + \dots + \alpha_n f_{\mathcal{D}_t}^n,
\]
Where $\alpha_1,\alpha_2,\dots,\alpha_n$ are the weights. The weights are regarded as hyper parameters. If the number of models $n$ is not too high, there is little risk of overfitting, and the hyper parameters can simply be directly fitted to minimise the quality of fit. In Section~\ref{sec: Results}, we use this method to combine the random forest and random Fourier features models.

\section{Results}\label{sec: Results}
We begin the results section by establishing a suitable choice for the number of time samples $|\mathcal{K}|$, as discussed in Section~\ref{sec: Quality of fit}. We are looking to satisfy two main conditions. First, $\widetilde{\mathcal{Q}}$ needs to be approximately normally distributed. As seen from Figure~\ref{fig: clt convergence}, the central limit theorem seems to hold for estimates of $\mathcal{E}$ with sample sizes $|\mathcal{K}|>50$. Using normality of $\mathcal{Q}$ and $\mathcal{E}$ means that two standard deviations of  correspond to the 5th percentile confidence intervals. Secondly, $|\mathcal{K}|$ needs to be sufficiently large for the error to be reasonably small. We went with $|\mathcal{K}|=500$, corresponding to $\mathrm{Var}(\widetilde{\mathcal{E}})\approx 0.1$ or a $1\%$ relative error in $\mathcal{Q}$, as mentioned in Section~\ref{sec: Quality of fit}.\\\\
The quality of fit measurements reported in table~\ref{tab: results} were all obtained using this sample size, and the reported uncertainty corresponds to two standard deviations, estimated according to~\eqref{eq: clt std}. As the table shows, the confidence bounds vary slightly depending on the model. The same sample size was also used for the differences $\Delta \mathcal{Q}$ between the quality of fit of the benchmarking models and the random Fourier features model listed in Table~\ref{tab: delta results}, as well as the hyper parameter grid searches shown in Figures~\ref{fig: div reg grid}~and~\ref{fig: gamma sigma grid}. The optimal hyperparameters for the random Fourier features model were $0.01$ for the Sobolev regularisation constant $\lambda$, $0.001$ for the divergence penalty $\eta$, $1.4$ for the exponent $\gamma$ and $2.25$ for the step size $\sigma$ in the proposal kernel in the adaptive Metropolis algorithm (see Algorithm \ref{alg: RFF}). Additionally, running the Metropolis algorithm was chosen for $B=500$ steps struck a good balance between convergence and computation time. We omit the hyper optimisation results for the benchmarking models, which were all performed using the same type of grid search methods.\\\\

\begin{figure}
    \centering
        \includegraphics[width = 0.6\linewidth]{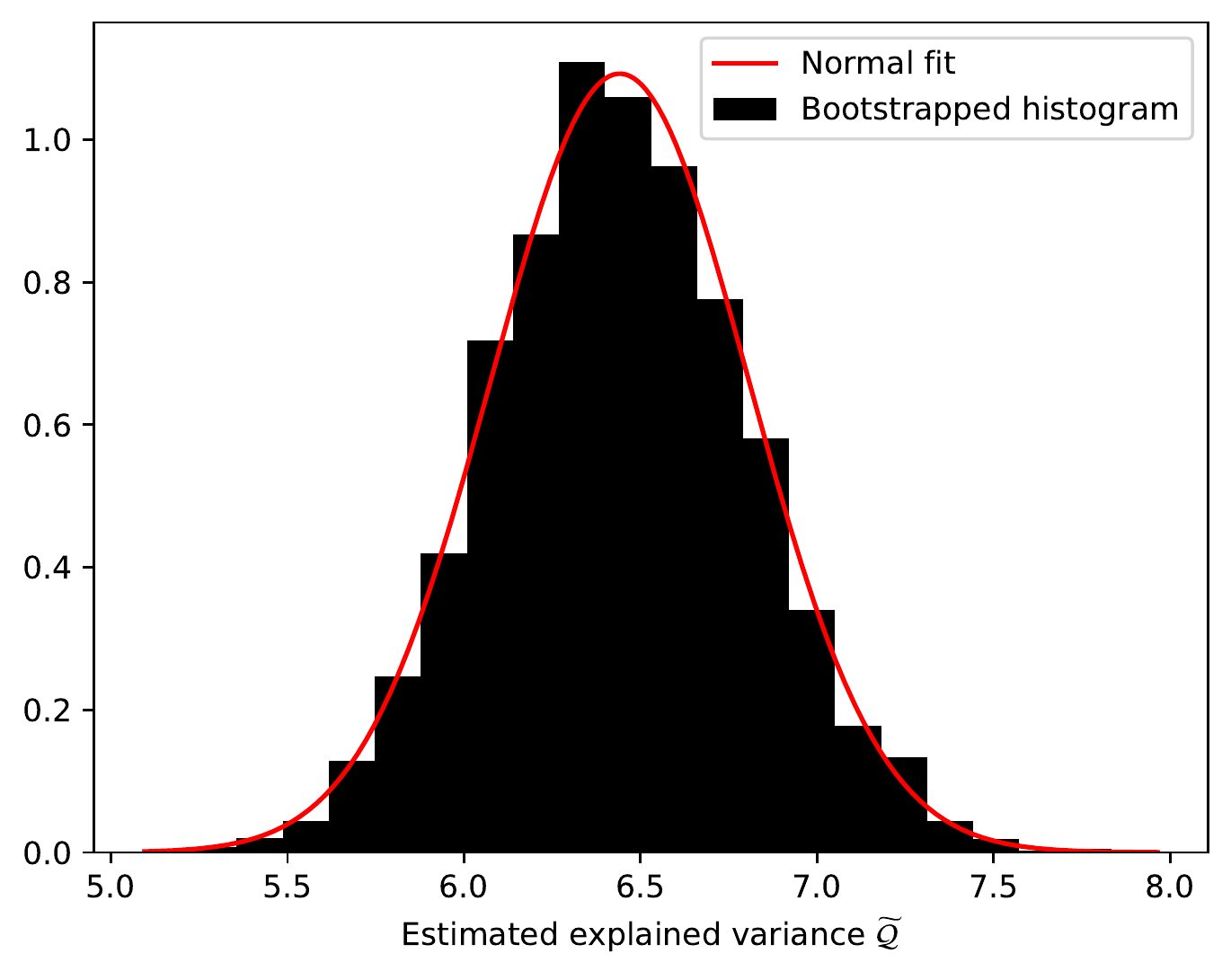}
        \captionsetup{width = 0.6\linewidth}
        \caption{A histogram of 5000 estimations of $\mathcal{Q}(f)=\frac{1}{|\mathcal{K}|}\sum_{k\in\mathcal{K}}\mathcal{Q}_k(f)$ with $|\mathcal{K}|=50$, bootstrapped from a total of 500 samples of $\mathcal{Q}_k$. The model $f$ is the fourier series, and the orange line is a fitted normal distribution.}
        \label{fig: clt convergence}
\end{figure}
\begin{figure}
    \centering
    \begin{minipage}{0.5\linewidth}
    \centering
    \includegraphics[width=\linewidth]{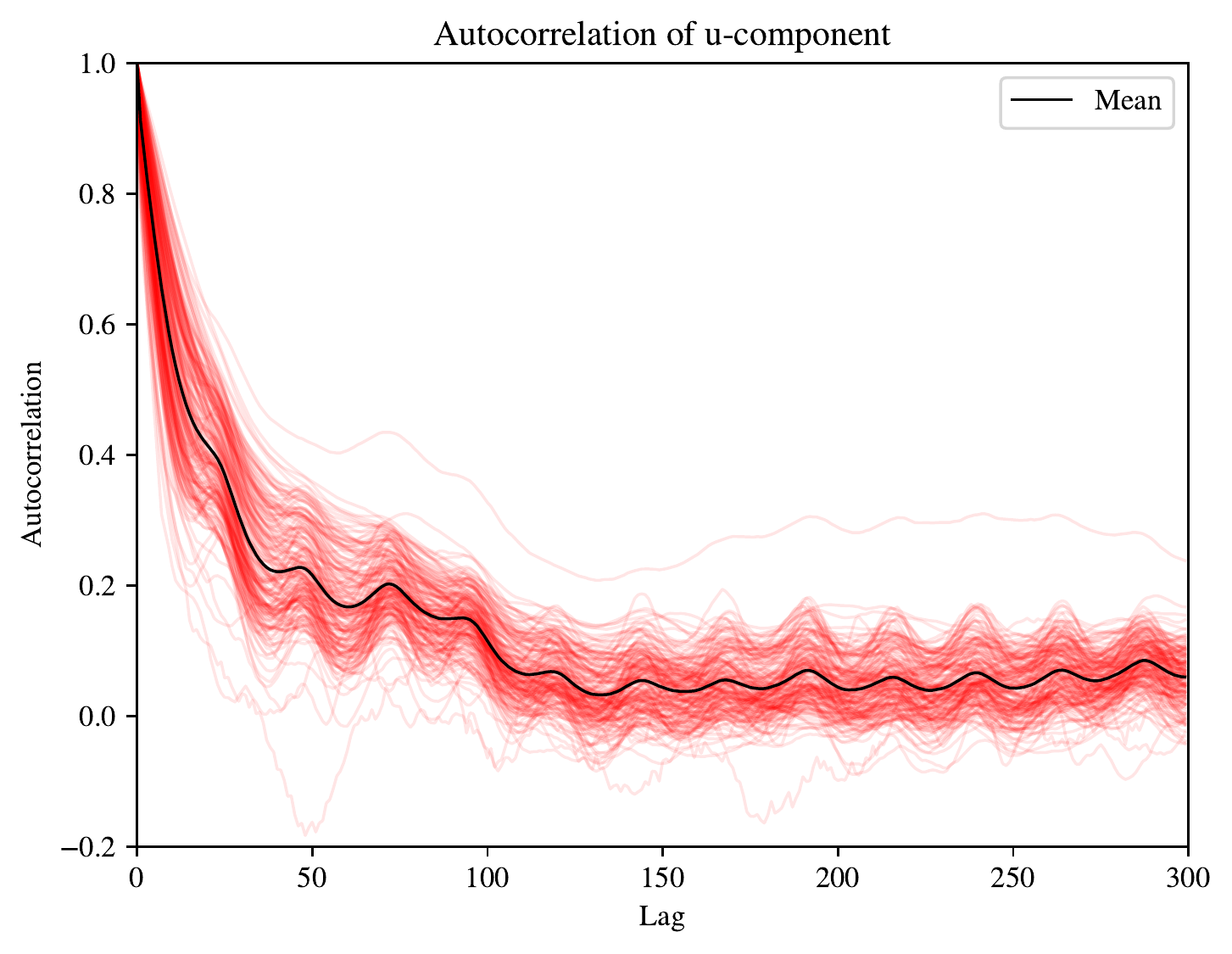}
    \captionsetup{width=0.9\linewidth}
    \caption{Autocorrelation of the east-west component $u$ of the wind $\pmb u$ for the available weather stations, up to a lag of 300 hours. Each red line represents the autocorrelation of one weather station.}    
    \label{fig: u-autocorr}
    \end{minipage}%
    \begin{minipage}{0.5\linewidth}
    \centering
    \includegraphics[width=\linewidth]{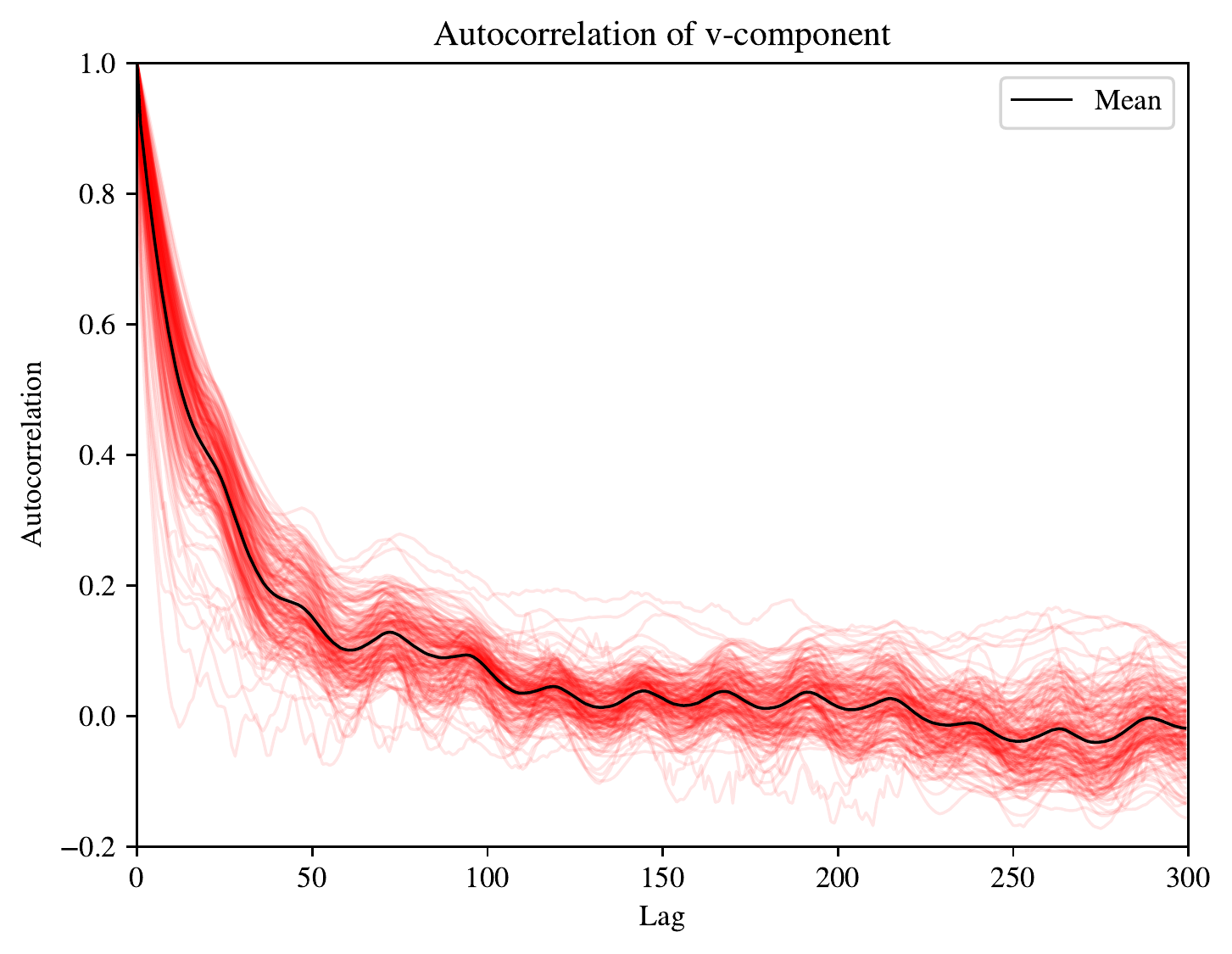}
    \captionsetup{width=0.9\linewidth}
    \caption{Autocorrelation of the north-south component $v$ of the wind $\pmb u$ for the available weather stations, up to a lag of 300 hours. Each red line represents the autocorrelation of one weather station.}    
    \label{fig: v-autocorr}
    \end{minipage}
\end{figure}
\renewcommand{\arraystretch}{1.4}
\begin{table}
    \centering
    \begin{tabular}{c c c}
        \hline
        Interpolation model &  $\widetilde{\mathcal{E}}\quad [1]$ &  $\widetilde{\mathcal{Q}}\quad  [m^2s^{-2}]$  \\
        \hline
Nearest neighbors & 0.628 $\pm$ 0.022 & 11.145 $\pm$ 0.386 \\
Inverse distance weighting & 0.407 $\pm$ 0.014 & 7.220 $\pm$ 0.250 \\
Universal Kriging & 0.388 $\pm$ 0.013 & 6.887 $\pm$ 0.235 \\
Random forest (RF) & 0.386 $\pm$ 0.013 & 6.862 $\pm$ 0.238 \\
Neural Network & 0.381 $\pm$ 0.013 & 6.762 $\pm$ 0.225 \\
Fourier series & 0.380 $\pm$ 0.013 & 6.740 $\pm$ 0.226 \\
\textbf{Random Fourier features (FF)} & 0.370 $\pm$ 0.012 & 6.569 $\pm$ 0.220 \\
FF and RF average & 0.357 $\pm$ 0.012 & 6.333 $\pm$ 0.212 \\
         \hline
    \end{tabular}
    \captionsetup{width = 0.9\linewidth}
    \caption{Quality of fit measurements $\mathcal{Q}$ and $\mathcal{E}$ with 5th percentile confidence intervals for a number of different models. The dimension is indicated at the top row, in square brackets.}
    \label{tab: results}
    \centering
    \begin{tabular}{c c c}
        \hline
        Interpolation model &  $\Delta\widetilde{\mathcal{E}}\quad [1]$ &  $\Delta \widetilde{\mathcal{Q}}\quad  [m^2s^{-2}]$  \\
        \hline
Nearest neighbours & 0.258 $\pm$ 0.010 & 4.576 $\pm$ 0.183 \\
Inverse distance weighting & 0.037 $\pm$ 0.003 & 0.651 $\pm$ 0.049 \\
Universal Kriging & 0.018 $\pm$ 0.002 & 0.318 $\pm$ 0.033 \\
Random forest & 0.017 $\pm$ 0.003 & 0.293 $\pm$ 0.060 \\
Neural Network & 0.011 $\pm$ 0.003 & 0.192 $\pm$ 0.056 \\
Fourier series & 0.010 $\pm$ 0.001 & 0.171 $\pm$ 0.017 \\
FF and RF average & -0.013 $\pm$ 0.002 & -0.236 $\pm$ 0.032 \\
         \hline
    \end{tabular}
    \captionsetup{width = 0.9\linewidth}
    \caption{Difference in quality of fit $\Delta\mathcal{Q}$ and $\Delta\mathcal{E}$ with 5th percentile confidence intervals for the benchmark models relative to the random Fourier features model (FF). The dimension is indicated at the top row, in square brackets. A positive number means that the given model is worse than random Fourier features.}
    \label{tab: delta results}
\end{table}
\begin{figure}
    \centering
    \begin{minipage}{0.5\textwidth}
        \includegraphics[width = 0.9\linewidth]{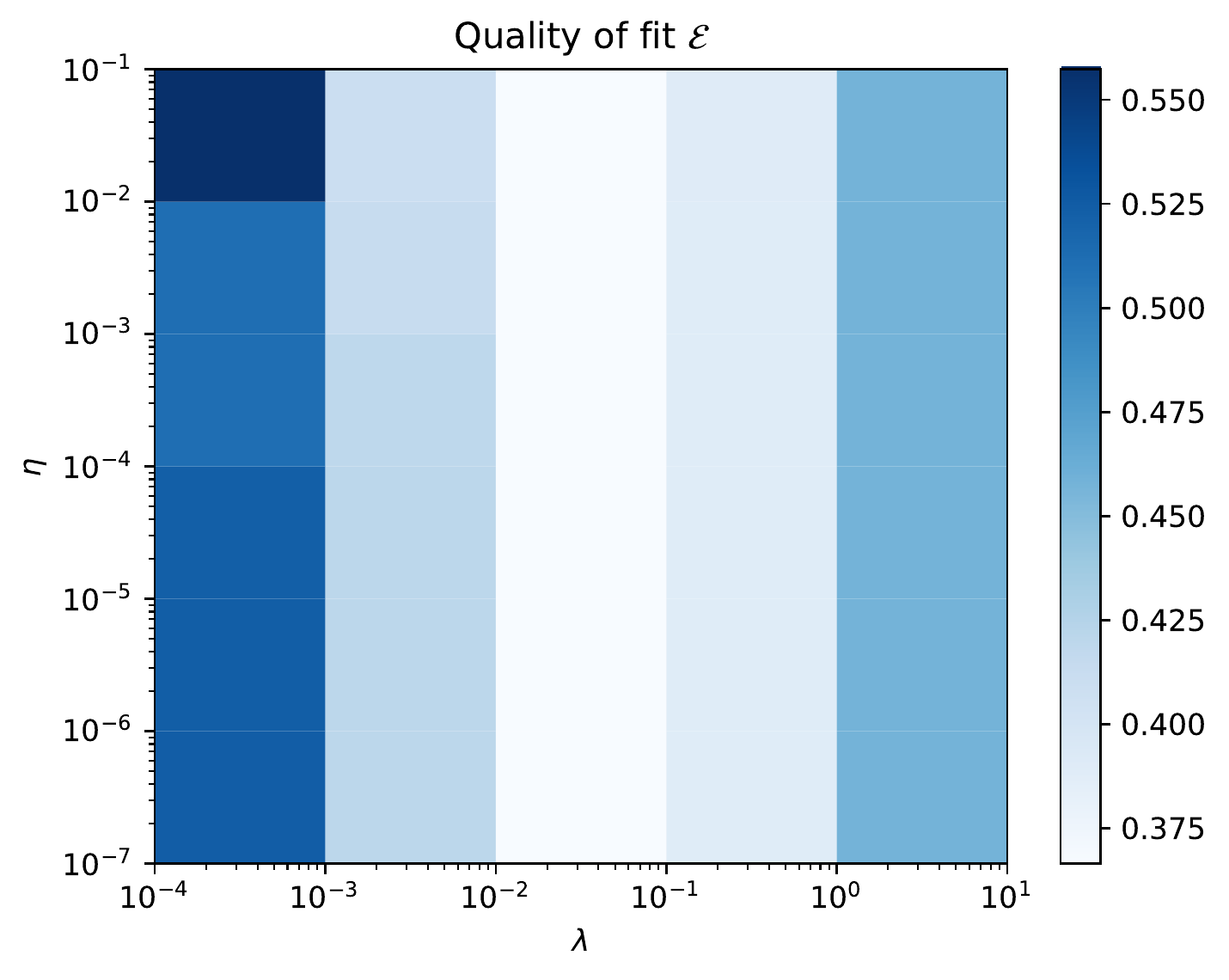}
        \captionsetup{width = 0.9\linewidth}
        \caption{Fraction of unexplained variance $\mathcal{E}$ in the random Fourier features method, as a function of the Sobolev penalty $\lambda$ and divergence penalty $\eta$ explained in~\ref{eq: loss}, with $\tau=4\times 10^7$, $\sigma=2.25$, $\gamma=1.25$.}
        \label{fig: div reg grid}
    \end{minipage}%
    \begin{minipage}{0.5\textwidth}
        \includegraphics[width = 0.9\linewidth]{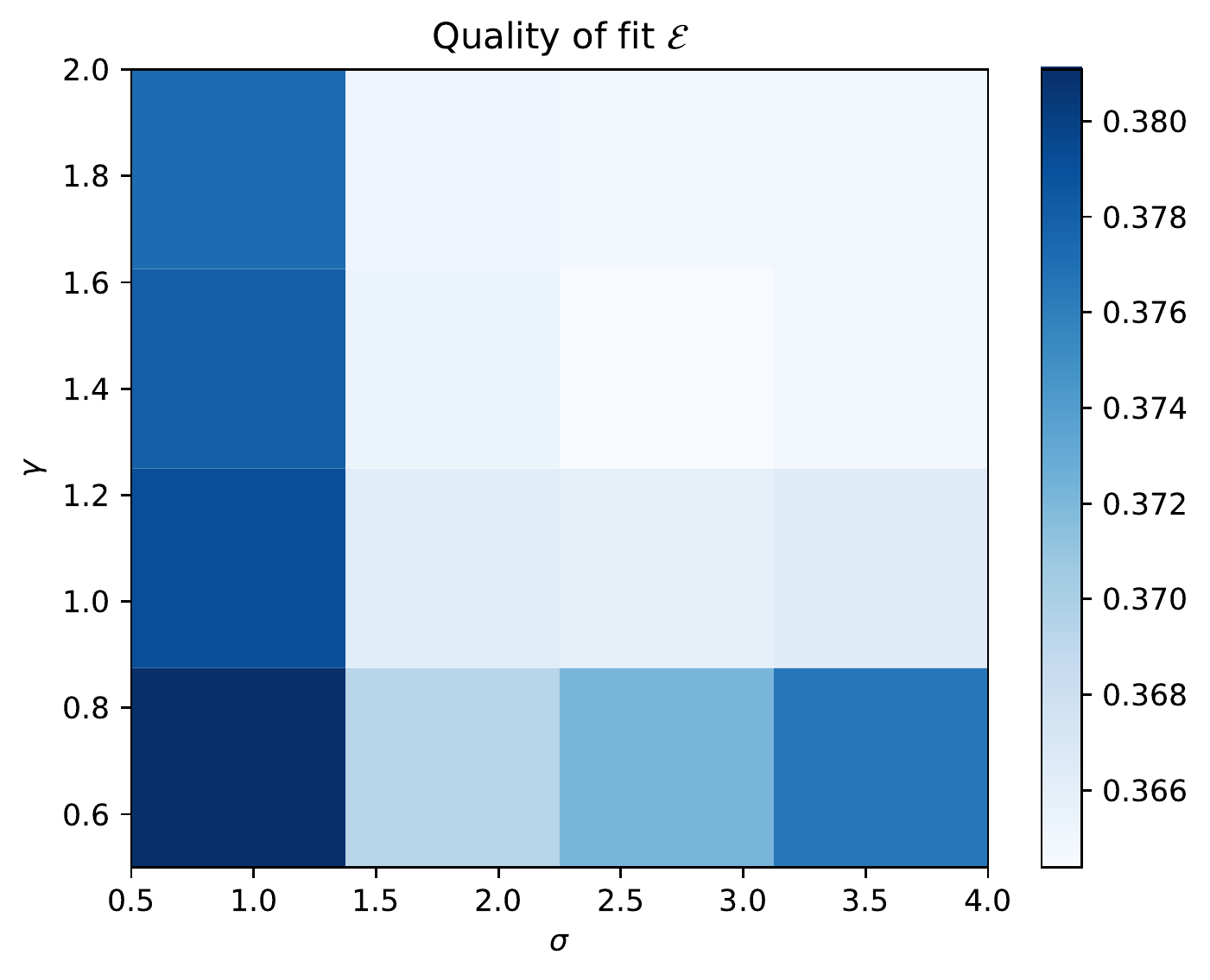}
        \captionsetup{width = 0.9\linewidth}
        \caption{Fraction of unexplained variance $\mathcal{E}$ as a function of the exponent $\gamma$ and step size $\sigma$ in the adaptive Metropolis algorithm, with $\tau=4\times 10^7$, $\eta=0.001$, $\lambda=0.01$.}
        \label{fig: gamma sigma grid}
    \end{minipage}
\end{figure}
\begin{figure}
    \centering
    \includegraphics[width = 0.8\linewidth]{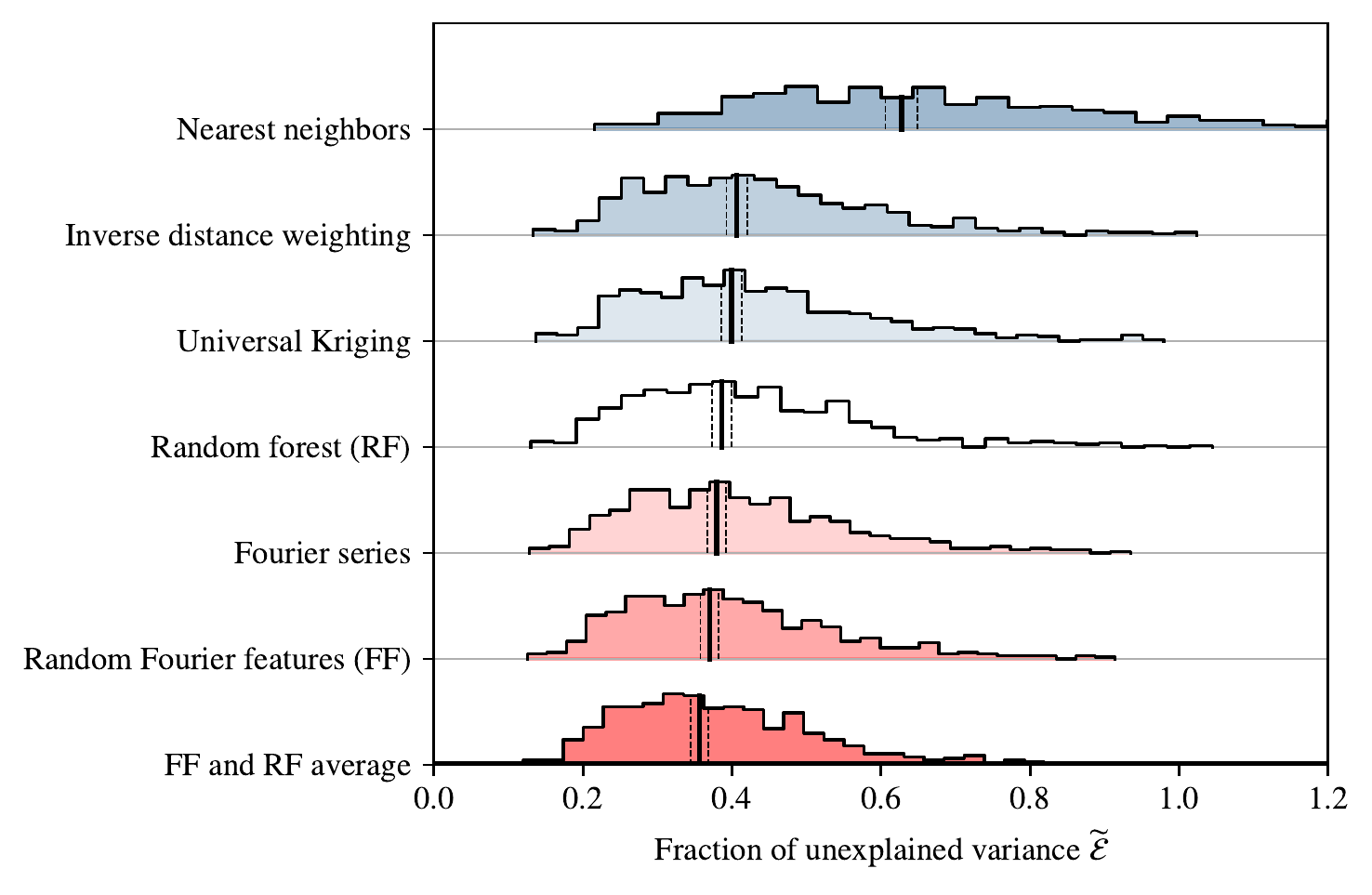}
    \caption{Ridge line plot showing histograms of the conditional fraction of unexplained variance $\widetilde{\mathcal{E}}_k(f) = \widetilde{\mathcal{Q}}_k(f)/\widetilde{\mathcal{Q}}(0), k\in\mathcal{K}$ with $|\mathcal{K}|=500$ independently sampled times, for different interpolation models $f$. The fraction of unexplained variance $\widetilde{\mathcal{E}}(f)=\frac{1}{|\mathcal{K}|}\sum_{k\in\mathcal{K}}\widetilde{\mathcal{E}}_k(f)$ is indicated as a black vertical line, and the uncertainty margin for $\widetilde{\mathcal{E}}(f)$ is indicated with dashed black lines.} 
    \label{fig: error histogram}
\end{figure}
\begin{figure}
    \centering
    \begin{minipage}{0.5\linewidth}
    \centering
        \includegraphics[width=0.8\linewidth]{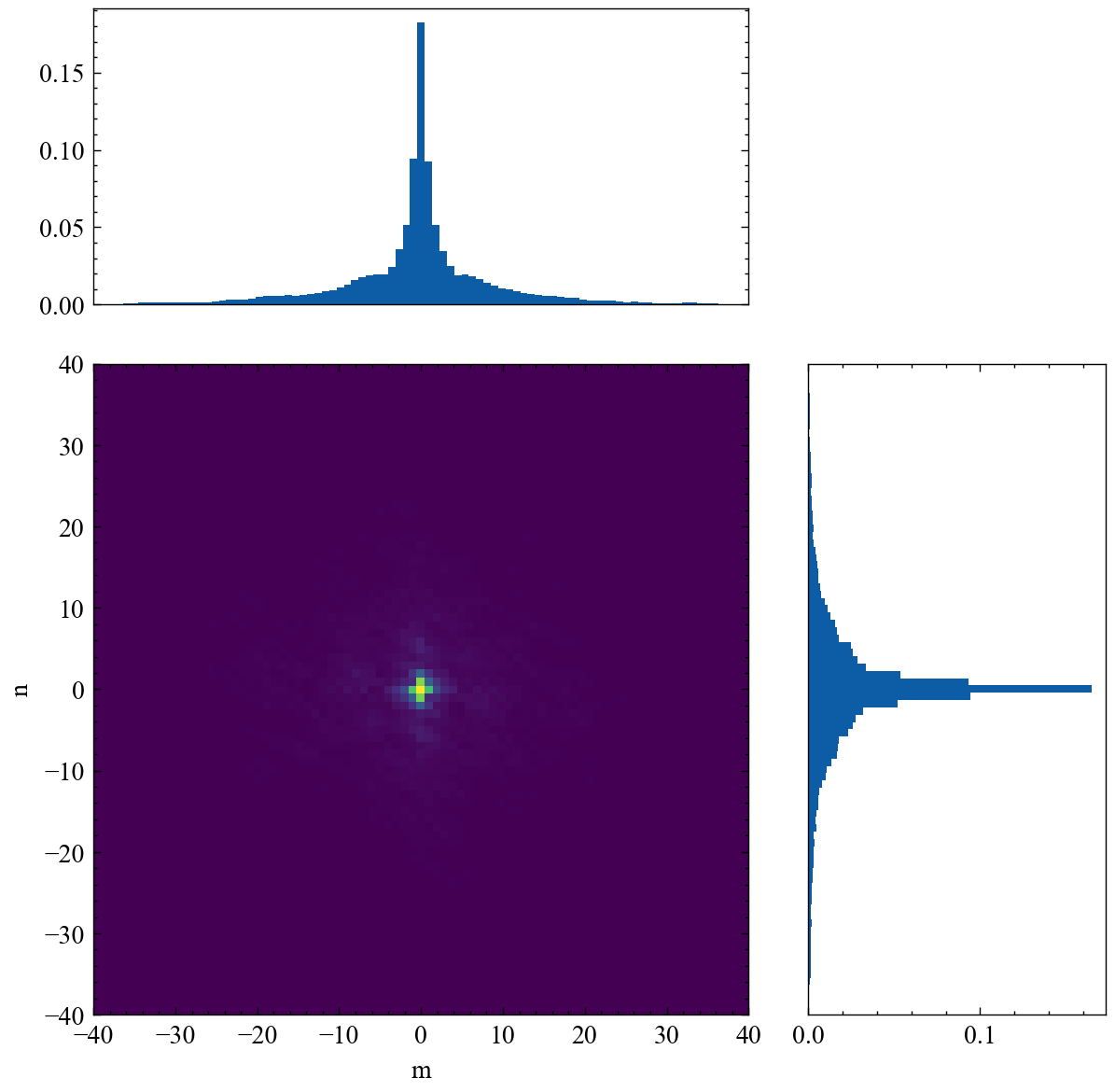}
        \captionsetup{width=0.8\linewidth}
        \caption{Approximation of the optimal sampling density $\rho(\omega)$ for the random Fourier features algorithm outlined in \ref{alg: RFF}, for measurement data from january 1st, 2018 at 6 am. The distribution was estimated by collecting all the frequencies from 1000 steps of the adaptive Metropolis algorithm into a histogram.}
        \label{fig: RFF histogram random}
    \end{minipage}%
    \begin{minipage}{0.5\linewidth}
    \centering
        \includegraphics[width=0.8\linewidth]{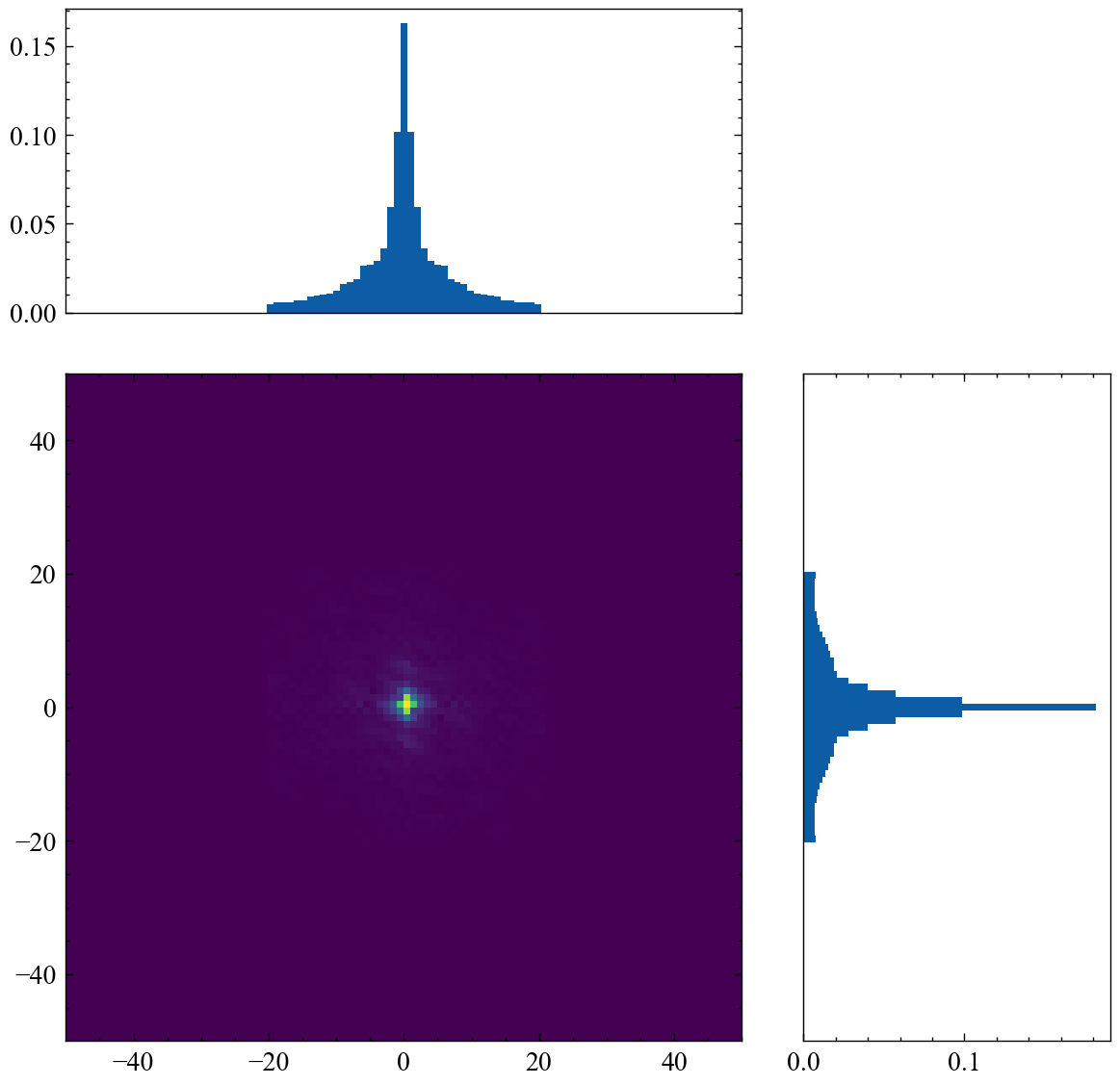}
        \captionsetup{width=0.8\linewidth}
        \caption{Heat map showing the magnitudes $||\beta||$ of the Fourier terms $\beta e^{i\omega^Tx}$ as a function of the frequencies $\omega$ for the Fourier series based model described in Section~\ref{sec: Fourier series}, for measurement data from January 1st, 2018 at 6 am. The support is a square grid with width 41, centered at the origin.}
        \label{fig: fourier histogram random}
    \end{minipage}
        \centering
        \includegraphics[width=\linewidth]{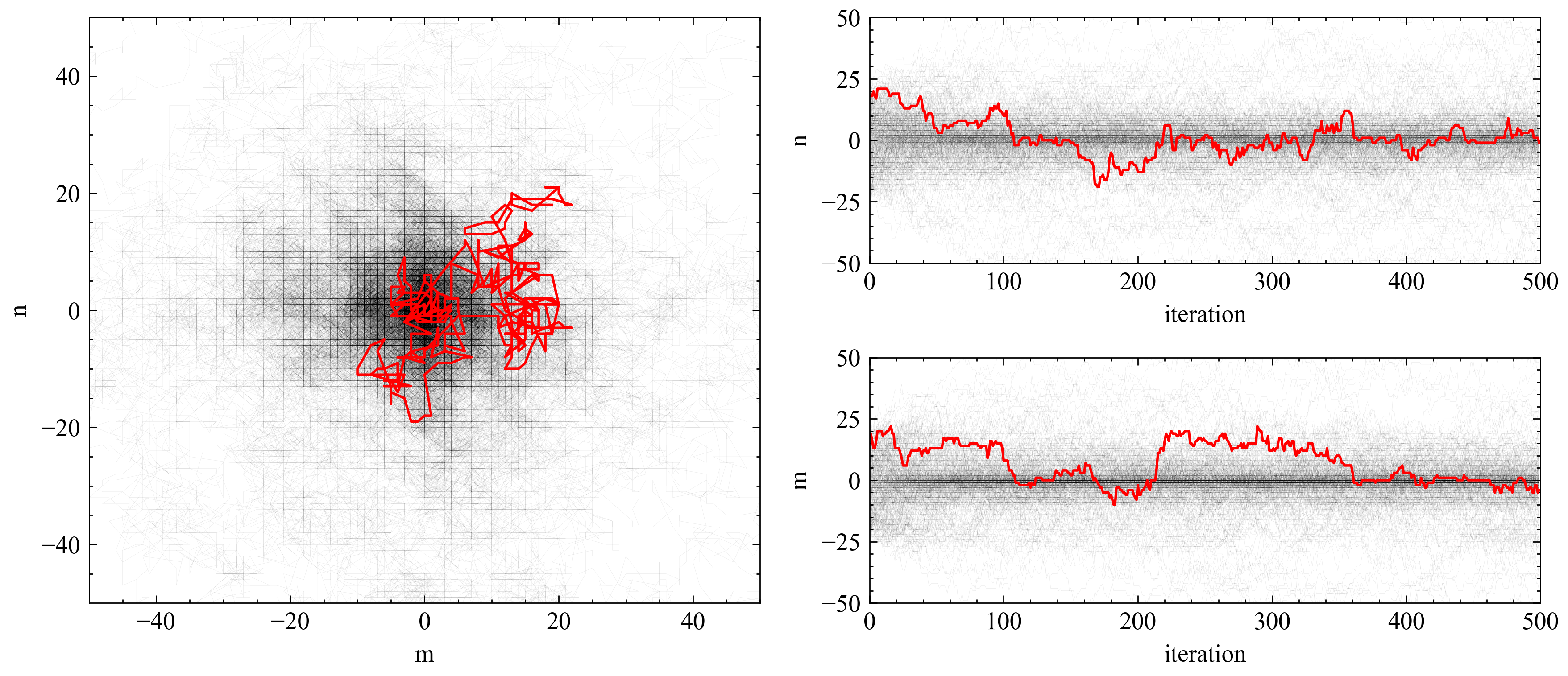}
        \captionsetup{width=\linewidth}
        \caption{Plot showing the trajectory of one frequency in \textcolor{red}{red}, sampled using the random Fourier features algorithm outlined in \ref{alg: RFF}. The algorithm was run on measurement data from January 1st, 2018 at 6 am. The remaining 400 trajectories are shown as black semi-transparent line plots. Some of the trajectories extend beyond the axis limits. A coordinate $(m,n)$ in the left plot represents an imaginary term $e^{i\pi\left(m\tfrac{x_1}{\tau_1}+n\tfrac{x_2}{\tau_2}\right)}$ in the Fourier series.}
        \label{fig: fourier trajectory}
\end{figure}
\clearpage
\section{Discussion}\label{sec: Discussion}
Table~\ref{tab: results} shows that the model consisting of an average between the random forest and random Fourier features model performed the best out of the tested models. Table~\ref{tab: delta results} consisting of the difference between quality of fit of the random Fourier features and remaining models indicates that this ordering is statistically significant, since none of the confidence intervals overlap with zero. Specifically, the transition from the fixed frequencies in the Fourier series based model to the randomly sampled frequencies of the random Fourier features results in a significant improvement. The numerical experiments found in~\cite{kammonen2020} indicate that the random Fourier features model wins out when more data is introduced. This is because the high frequency details can be captured by exploring remote parts of the frequency domain, whereas a fixed grid of frequencies centered on the origin cannot.\\\\
The success of the random forest-random Fourier features average hints that there might be more potential for improving the results using similar types of mixture models. Nevertheless, it is clear that the random Fourier features model is competitive in comparison to the benchmark models. A common argument for favouring statistical models such as Kriging before machine learning methods is easy access to precise error analysis, given that the prerequisite assumptions discussed in Section~\ref{sec: Kriging}. Whether or not this is worth trading in exchange for higher accuracy depends on the situation. Furthermore, the random Fourier features model has the upside of being easy to manipulate once it has been trained. It can be efficiently differentiated and integrated for analysis of physical quantities such as divergence, vorticity and energy.\\\\
A fraction of unexplained variance $0.36$ is not immediately recognised as a good result. For example, the works \cite{Erichson2020, Callaham2019, jin2020} achieve an unexplained variance in the order of $0.001$ modelling 2D fluid flow around a cylinder as well as sea mean temperatures. The main difference is that our data is extremely sparse, that there are no high resolution grids to train the models on, and that wind data in particular is known for high variance over time and space. The authors of \cite{Erichson2020} point out \cite[p.~18]{Erichson2020} that their model generalises poorly to unseen data  if the flow is non-stationary. It stands to reason that for their models to achieve the same efficiency on the sparse wind data used in this report, higher resolution in-sample measurements or simulations are required for training. That being said, we know from Figure~\ref{fig: u-autocorr} that the wind measurements are highly correlated over time, with clear seasonality. The models used in \cite{Erichson2020, Callaham2019, jin2020} are trained on multiple time samples whereas the spatial interpolation models used in our work only use the time aspect for hyperparameter optimisation. Therefore, extending the spatial interpolation models to take time into account could improve the results.\\\\
Another way of improving the results is to simply increase the quality or number of measurements by incorporating weather stations from nearby regions such as the Baltic sea, Finland, Norway and Denmark. The random Fourier features model is constructed to efficiently find high frequency details in the target function, which makes it suitable when increasing the resolution of the data. Additionally, using more features such as terrain convexity, terrain slope, terrain roughness, altitude, air pressure, temperature and humidity might allow for more accurate predictions (see \cite{appelhans2015}). We found that the wind speed was highly correlated with altitude as well as closeness to coastal areas, and adding altitudes to the feature space of the random forest resulted in a significant increase in accuracy. \\\\
As seen in Figure~\ref{fig: div reg grid}, the divergence penalty does not have much influence over the results. Albeit a compelling idea, it makes sense that penalising the divergence on a 2D slice of the wind field would not improve the results, since wind flow is not necessarily parallel with the ground. We suggest two alternative approaches that could produce significant improvements. Firstly, by incorporating altitudes as a third feature in the random Fourier features model we go from a 2D to a 3D setting in which zero divergence is a better approximation of reality. An alternative approach is to model the wind flow in two layers, and to couple the two layers by penalising a weighted sum of the divergence. More advanced methods can possibly incorporate no-slip or slipping boundary conditions on the ground surface, like the methods explored in~\cite{tempone99}.\\
\section{Conclusion}
In this report, we explored the potential for wind field reconstruction with sparse data using interpolation models. In particular, we investigated the random Fourier features model and compared it to popular statistical interpolation models such as Kriging, as well as modern machine learning methods such as random forests and Neural networks. Drawing from the work~\cite{kammonen2020}, we derived an optimal density for the Fourier frequencies and devised an adaptive Metropolis algorithm for sampling from this density. We showed that random Fourier features is competitive with respect to a time-space average of the square error and provided some ways to improve the results such as including terrain specific features.\\

\clearpage
\section{Acknowledgements}
The authors of this report would like to thank Prof. Anders Szepessy at KTH for his support and feedback. Furthermore, the work of Dmitry Kabanov, Luis Espath and Andreas Enblom in data processing and programming was integral in realising the project. Jonas Kiessling and Raúl Tempone were partially supported by the KAUST Office of Sponsored Research (OSR) under Award numbers URF/1/2281-01-01, URF/1/2584-01-01 in the KAUST Competitive Research Grants Program Round 8, and the Alexander von Humboldt Foundation, through the Alexander von Humboldt Professorship award. Emanuel Ström was supported by the KAUST Visiting Student Research Program (VSRP).

\section{Appendix}
In this proposition, we derive an upper bound for the minimum of the expected loss $\mathrm{E}\left[\ell(\beta(\pmb x),\pmb u)\right]$ for the Fourier features model. We assume the standard regression setting where $\pmb u = f(\pmb x) + \epsilon$ and $\epsilon\in\R^2$ is independent of $\pmb x$, and $\mathrm{E}[||\epsilon||^2]=\sigma^2$. Let $\hat f(\omega)$ define the Fourier coefficients of $f$ and suppose for simplicity that $f$ is defined on the domain $X = [0,2\pi]\times [0,2\pi]$. That is, $f$ can be expressed as the Fourier series
\[
    f(\pmb x) = \frac{1}{2\pi}\sum_{\omega\in\Z^2} \hat f(\omega)e^{i\omega\cdot \pmb x}.
\]
For the random Fourier features model, we choose 
\[
    \beta(\pmb x) = \sum_{k=1}^K \beta_k e^{i\omega_k \cdot \pmb x},
\]
where  $\pmb\beta = (\beta_1, \beta_2, \dots,\beta_K)$ are the complex-valued two-dimensional coefficients and $\pmb \omega = (\omega_1, \omega_2, \dots ,\omega_K)$ are independent and identically distributed according to a discrete distribution $\rho\colon \Z^2\to [0,\infty)$. The loss function is defined as
\[
    \ell(\beta(\pmb x), \pmb u) = ||\beta(\pmb x) - \pmb u||^2 + \lambda ||\mathcal{L}\beta ||^2,
\]
where
\[
    ||\mathcal{L}\beta ||^2 = \int_{X}\overline{ \mathcal{L}\beta}(\pmb x)\mathcal{L}\beta(\pmb x)\mathrm{d}\pmb x,
\]
 and $\mathcal{L} = \sum_{m=1}^M c_m\partial_1^{\alpha_{m,1}}\partial_2^{\alpha_{m,2}}$ is a linear differential operator with derivatives of at most order $d$ (i.e. $\alpha_{m,1}+\alpha_{m,2} \leq d$), for example $\mathcal{L}\beta = \nabla \cdot \beta = (\partial_1 + \partial_2)\beta$. The distribution $\rho$ is assumed to exist in a family a family $\mathcal{P}$  of discrete distributions:
\[
    \mathcal{P} :=\left\{\rho\colon \Z^2\to (0,\infty)\quad\Big|\quad \rho(\omega)>0 \quad \text{and} \quad \mathrm{E}[|\omega_i|^{m}]< C, \quad 0\leq m\leq 4d\quad i=1,2\right\}
\]
 where $C$ is a positive real number. Thus, $\rho\in\mathcal{P}$ is strictly positive and has uniformly bounded moments $\sum_{\omega\in \Z^2} |\omega_i|^{m}\rho(\omega)$ of degree up to $4d$.  Lastly, we make the assumption that
 \begin{equation}\frac{||\hat f(\omega)||}{\sum_{\omega' \in\Z^2}||\hat f(\omega')||}\in
 \mathcal{P}.
 \label{eq: f in P}
 \end{equation}
 Which means that $f(\pmb x)$ has to be a member of the Sobolev space $W^{4d, 2}(X)$.
 \begin{proposition}
 In the above setting, the following holds:
 \begin{enumerate}[(a)]
     \item The minimum of $\mathrm{E}[\ell(\beta(\pmb x),\pmb u)]$ with respect to the coefficients $\pmb\beta$ can be bounded:
 \[
    \mathrm{E}\left[\min_{\pmb\beta\in \C^{2K}}\mathrm{E}\left[\ell(\beta(\pmb x),\pmb u)\mid \pmb \omega\right]\right] \leq \frac{1+\lambda \overline{C}}{(2\pi)^2 K}\sqrt{\mathrm{E}\left[\frac{||\hat f(\omega)||^4}{\rho(\omega)^4}\right]} + \sigma^2 - \frac{1}{K}\mathrm{E}[|| f(\pmb x) ||^2], 
 \]
 where $\overline{C} > 0$.
 \item Furthermore, this upper bound is minimised by the distribution
 \[
    \rho(\omega) = \frac{||\hat f(\omega)||}{\sum_{\omega'\in\Z^2}||\hat f(\omega')||},\quad \omega \in \Z^2.
 \]
 \end{enumerate}
 
\end{proposition}
\begin{proof} We divide the problem into part (a) and part (b) of the proposition.
\begin{enumerate}[(a)]
\item Let $\omega :=\omega_1, \beta := \beta_1$ for the sake of brevity, and define
\[
    \beta_k = \frac{\hat f(\omega_k)}{2\pi K \rho(\omega_k)}, \quad k = 1,2,\dots ,K.
\]
Using this definition, $\beta(\pmb x)$ is unbiased given $\pmb x$. To show this, we use the iid property of $\pmb\omega$:
\begin{align*}
\mathrm{E}\left[ \sum_{k=1}^K\beta_k e^{i\omega_k\cdot \pmb x}\;\Bigg|\; \pmb x\right] &=\mathrm{E}\left[K \beta_1 e^{i\omega_1\cdot \pmb x}\mid \pmb x\right] \\
&=\sum_{\omega\in\Z^2}K\frac{\hat f(\omega)}{2\pi K\rho(\omega)}e^{i\omega\cdot \pmb x}\rho(\omega)\\
&=\frac{1}{2\pi}\sum_{\omega'\in\Z^2}\hat f(\omega)e^{i\omega\cdot \pmb x} = f(\pmb x).
\end{align*}
Coincidentally, the expected square error can be simplified using independence of the residuals $\epsilon$ and the data $\pmb x$:
\begin{align*}
    \mathrm{E}\left[||\beta(\pmb x) - \pmb u ||^2\mid \pmb x\right] &= \mathrm{E}\left[||\beta(\pmb x) - (f(\pmb x) + \epsilon)||^2\mid \pmb x\right] \\&=
    \mathrm{E}\left[||(\beta(\pmb x) - f(\pmb x)) - \epsilon||^2\mid \pmb x\right] \\ &=\mathrm{E}\left[||\beta(\pmb x) - f(\pmb x)||^2\mid \pmb x\right] + \sigma^2.
\end{align*} 
Which shows that the expected square error is exactly the variance of $\beta$, plus the mean square of the noise. The variance can be simplified further since the frequencies $\omega_k$ are assumed independent:
\begin{align*}
    \mathrm{E}\left[\left|\left|\sum_{k=1}^K\beta_ke^{i\omega_k \cdot \pmb x}-f(\pmb x)\right|\right|^2\right]
    &= \frac{1}{K}\mathrm{E}\left[\left|\left|K\beta e^{i\omega \cdot \pmb x}-f(\pmb x)\right|\right|^2\right]\\
    &=K\mathrm{E}\left[||\beta||^2 \right]-\frac{1}{K}\mathrm{E}\left[||f(\pmb x)||^2\right]\\
    &=\frac{1}{(2\pi)^2K}\mathrm{E}\left[\frac{||\hat f(\omega)||^2}{\rho(\omega)^2}\right] - \frac{1}{K}\mathrm{E}\left[||f(\pmb x)||^2\right]\\
    &\leq \frac{1}{(2\pi)^2K}\sqrt{\mathrm{E}\left[\frac{||\hat f(\omega)||^4}{\rho(\omega)^4}\right]} - \frac{1}{K}\mathrm{E}\left[||f(\pmb x)||^2\right].
\end{align*}
The last step is due to the Jensen inequality. Now, turn to the penalty term containing the linear operator $\mathcal{L}$. Applying $\mathcal{L}$ to $\beta(\pmb x)$ is equivalent to multiplying each term $\beta_{kj}e^{i\omega_k x}, j=1,2$ of the Fourier series with $\ell_j(\omega_k) = \sum_{m=1}^M c_m (i\omega_{k,1})^{\alpha_{m,1}}(i\omega_{k,2})^{\alpha_{m,2}}$, a multivariate polynomial in $\omega_k$ of degree $d$. Define $r(\omega) = |\ell_1(\omega)|^2 + |\ell_2(\omega)|^2$. Note that $r$ has degree $2d$. It follows that
\begin{multline*}
    ||\mathcal{L}\beta||^2 = \int_{X}||\mathcal{L}\beta(x)||^2 \mathrm{d}x = \int_{X}\left|\sum_{k=1}^K (\ell_1(\omega_k)\beta_{k1} + \ell_2(\omega_k)\beta_{k2}) e^{i\omega_k\pmb x}\right|^2 \mathrm{d}\pmb{x} \leq\\ \leq \sum_{k=1}^K|\ell_1(\omega_k)\beta_{k1} + \ell_2(\omega_k)\beta_{k2}|^2\int_X 1\mathrm{d}x = \sum_{k=1}^K(2\pi)^2|\ell_1(\omega_k)\beta_{k1} + \ell_2(\omega_k)\beta_{k2}|^2\\\leq \sum_{k=1}^K r(\omega_k)||2\pi \beta_k||^2,
\end{multline*}
where we used the triangle inequality combined with the size $(2\pi)^2$ of the region $X=[0,2\pi]\times[0,2\pi]$. Taking the expectation, using independence of its elements as well as the Cauchy-Schwarz inequality we get:
\begin{multline*}
    \mathrm{E}\left[||\mathcal{L}\beta||^2\right] \leq \mathrm{E}\left[\sum_{k=1}^Kr(\omega_k)||2\pi\beta_k||^2\right]=\sum_{k=1}^K \mathrm{E}\left[r(\omega_k)||2\pi\beta_k||^2\right]=K \mathrm{E}\left[r(\omega)||2\pi\beta||^2\right]\\
    =K\mathrm{E}\left[r(\omega)\left|\left|2\pi\frac{\hat f(\omega)}{2\pi K \rho(\omega)}\right|\right|^2\right] \leq \frac{1}{K}\sqrt{\mathrm{E}\left[r(\omega)^2\right]\mathrm{E}\left[\frac{||\hat f(\omega)||^4}{\rho(\omega)^4}\right]}
\end{multline*}
The multivariate polynomial $r(\omega)^2 := \sum_{m=1}^{M'}r_m \omega_1^{\alpha_{m,1}}\omega_2^{\alpha_{m,2}}$ is of order $4d$, i.e. $\alpha_{m,1}+\alpha_{m,2}\leq 4d$. For each term $m$, define $\gamma_m := \frac{\alpha_{m,1}}{\alpha_{m,1}+\alpha_{m,2}}\in (0,1)$. By assumption, $\rho$ lies in $\mathcal{P}$, and therefore, the expectation $\mathrm{E}\left[r(\omega)^2\right]$ can be uniformly bounded as follows:
\begin{align*}
    \mathrm{E}\left[r(\omega)^2\right] &= \sum_{m=1}^{M'}r_m \mathrm{E}\left[\omega_1^{\alpha_{m,1}}\omega_2^{\alpha_{m,2}}\right] \leq \sum_{m=1}^{M'}|r_m| \mathrm{E}\left[|\omega_1|^{(\alpha_{m,1}+\alpha_{m,2})\gamma_m}|\omega_2|^{(\alpha_{m,1}+\alpha_{m,2})(1-\gamma_m)}\right]\\
    &\leq  \sum_{m=1}^{M'}|r_m| \underbrace{\left({\mathrm{E}\left[|\omega_1|^{\alpha_{m,1}+\alpha_{m,2}}\right]}^{\gamma_m}{\mathrm{E}\left[|\omega_2|^{\alpha_{m,1}+\alpha_{m,2}}\right]}^{1-\gamma_m}\right)}_{\leq C^{\gamma_m}C^{1-\gamma_m} = C} \leq  C\sum_{m=1}^{M'}|r_m| := \frac{\overline{C}}{(2\pi)^2}.
\end{align*}
In the above calculations, we used Hölder's inequality: $\mathrm{E}[XY]\leq (\mathrm{E}[|X|^p])^{\frac{1}{p}}(\mathrm{E}[Y^q]^{\frac{1}{q}})$ for any $\frac{1}{p}+\frac{1}{q}=1$, termwise with $p = \frac{1}{\gamma_m}$. Collecting all of our results, the expected loss is bounded by
\begin{align*}
    \mathrm{E}\left[\min_{\pmb\beta}\mathrm{E}\left[||\beta(\pmb x)-\pmb u||^2+||\mathcal{L}\beta||^2 \mid \pmb \omega\right]\right] &\leq \frac{1+\lambda \overline{C}}{(2\pi)^2K}\sqrt{\mathrm{E}\left[\frac{||\hat f(\omega)||^4}{\rho(\omega)^4}\right]} -\frac{1}{K}\mathrm{E}[||f(\pmb x)||^2]+\sigma^2.
\end{align*}
\item To derive an optimal choice for $\rho$, we seek to minimise the expression inside the root. We can redefine this problem in terms of a non-normalised function $p$ such that $\rho(\omega) = p(\omega)/\sum_{\Z^2}p(\omega')$:
\[
    \text{minimise}\quad  \left(\sum_{\Z^2}\frac{||\hat f(\omega)||^4}{p(\omega)^3}\right)\cdot\left(\sum_{\Z^2}p(\omega)\right)^3.
\]
First, define a real-valued function $H(\epsilon)$ where $\epsilon$ is a real number close to zero:
\[
    H(\epsilon) = \left(\sum_{\Z^2}\frac{||\hat f(\omega)||^4}{(p(\omega)+\epsilon \delta(\omega))^3}\right)\cdot\left(\sum_{\Z^2}p(\omega)+\epsilon\delta(\omega)\right)^3,
\]
where $\delta$ is a small arbitrary variation of $p$. Next, seek a solution $p$ to $H'(0) = 0$. 
\begin{multline*}
    H'(0) = \left(\sum_{\Z^2}-3\frac{||\hat f(\omega)||^4}{p(\omega)^4}\delta(\omega)\right)\cdot\overbrace{\left(\sum_{\Z^2}p(\omega)\right)^3}^{c_1}+ \\ \underbrace{\left(\sum_{\Z^2}\frac{||\hat f(\omega)||^4}{p(\omega)^3}\right)\cdot\left(\sum_{\Z^2}p(\omega)\right)^2}_{c_2}\sum_{\Z^2}\delta(\omega) = 0.
\end{multline*}
Defining the constants $c_1,c_2$ as above, the equation can be rewritten as 
\[
    \sum_{\Z^2}\left(\frac{||\hat f(\omega)||^4}{p(\omega)^4} - \frac{c_2}{3c_1}\right)\delta(\omega) = 0
\]
Since $\delta(\omega)$ is arbitrary here, the expression inside the sum must be zero. Thus,
\[
    \left(\frac{||\hat f(\omega)||^4}{p(\omega)^4} - \frac{c_2}{3c_1}\right) = 0 \iff p(\omega) = \sqrt[4]{\frac{3c_1}{c_2}}||\hat f(\omega)||
\]
Hence, the optimal $\rho$ is 
\[
    \rho(\omega) = \frac{p(\omega)}{\sum_{\Z^2}p(\omega')} =\frac{\sqrt[4]{\frac{3c_1}{c_2}}||\hat f(\omega)||}{\sum_{\Z^2}\sqrt[4]{\frac{3c_1}{c_2}}||\hat f(\omega')||} =  \frac{||\hat f (\omega)||}{\sum_{\Z^2}||\hat f(\omega')||}.
\]
Lastly, some straight forward calculations show the optimisation with respect to $\rho$ is a convex optimisation problem (that is, $\mathcal{P}$ is a convex set and $\mathrm{E}\left[\frac{||\hat f(\omega)||^4}{\rho(\omega)^4}\right]$ is a convex function with respect to $\rho$). But then, the above derived local minimum must also be a global minimum.
\end{enumerate}

\end{proof}

\clearpage
\bibliography{bibliography}

\begin{thebibliography}{10}

\bibitem{tempone99}
R.~F. Tempone, ``Approximation and interpolation of divergence free flows,''
  1999.
\newblock Tesis de maestría. Universidad de la República (Uruguay). Facultad
  de Ingeniería.

\bibitem{kammonen2020}
A.~Kammonen, J.~Kiessling, P.~Plecháč, M.~Sandberg, and A.~Szepessy,
  ``Adaptive random fourier features with metropolis sampling,'' {\em
  Foundations of Data Science}, vol.~2, no.~2639-8001 2020 3 309, p.~309, 2020.

\bibitem{carta2009}
J.~A. Carta, P.~Ramírez, and S.~Velázquez, ``A review of wind speed
  probability distributions used in wind energy analysis case studies in the
  canary islands,'' {\em Renewable and Sustainable Energy Reviews}, 2008.

\bibitem{Luo2008}
W.~Luo, M.~C. Taylor, and S.~R. Parker, ``A comparison of spatial interpolation
  methods to estimate continuous wind speed surfaces using irregularly
  distributed data from england and wales,'' {\em International Journal of
  Climatology}, vol.~28, no.~7, pp.~947--959, 2008.

\bibitem{Callaham2019}
J.~L. Callaham, K.~Maeda, and S.~L. Brunton, ``Robust flow reconstruction from
  limited measurements via sparse representation,'' {\em Physical Review
  Fluids}, vol.~4, Oct 2019.

\bibitem{jin2020}
X.~Jin, S.~Laima, W.-L. Chen, and H.~Li, ``Time-resolved reconstruction of flow
  field around a circular cylinder by recurrent neural networks based on
  non-time-resolved particle image velocimetry measurements,'' {\em Experiments
  in Fluids}, vol.~61, 04 2020.

\bibitem{Erichson2020}
N.~Erichson, L.~Mathelin, Z.~Yao, S.~Brunton, M.~Mahoney, and J.~Kutz,
  ``Shallow neural networks for fluid flow reconstruction with limited
  sensors,'' {\em Proceedings of the Royal Society A: Mathematical, Physical
  and Engineering Sciences}, vol.~476, p.~20200097, 06 2020.

\bibitem{hengl2018}
T.~Hengl, M.~Nussbaum, M.~N. Wright, G.~B.~M. Heuvelink, and B.~Gräler,
  ``Random forest as a generic framework for predictive modeling of spatial and
  spatio-temporal variables,'' 2018.

\bibitem{appelhans2015}
T.~Appelhans, E.~Mwangomo, H.~Douglas~R, A.~Hemp, and T.~Nauss, ``Evaluating
  machine learning approaches for the interpolation of monthly air temperature
  at mt. kilimanjaro, tanzania,'' {\em Spatial Statistics}, 2015.

\bibitem{erxleben2002}
J.~Erxleben, K.~Elder, and R.~Davis, ``Comparison of spatial interpolation
  methods for estimating snow distribution in the colorado rocky mountains,''
  {\em Hydrological Processes}, 2002.

\bibitem{cellura2008}
M.~Cellura, G.~Cirrincione, A.~Marvuglia, and A.~Miraoui, ``Wind speed spatial
  estimation for energy planning in sicily: Introduction and statistical
  analysis,'' {\em Renewable Energy}, 2008.

\bibitem{jung2015}
C.~Jung and D.~Schindler, ``Statistical modeling of near-surface wind speed: A
  case study from baden-wuerttemberg (southwest germany),'' {\em Austin Journal
  of Earth Science}, 2015.

\bibitem{JinBook2008}
J.~Li, {\em A Review of Spatial Interpolation Methods for Environmental
  Scientists}.
\newblock 01 2008.

\bibitem{SMHI}
``Ladda ner meterologiska observationer, vindhastighet och vindriktning.''
  \\\url{https://www.smhi.se/data/meteorologi/ladda-ner-meteorologiska-observationer/#param=wind,stations=all}.
\newblock Accessed: 2020-07-30.

\bibitem{oroberts2001}
G.~O. Roberts and J.~S. Rosenthal, ``Optimal scaling for various
  metropolis-hastings algorithms,'' {\em Statistical Science. 16(4):351–367,
  11 2001}, 2001.

\bibitem{unpublished_luis}
L.~Espath, J.~Kiessling, and D.~Kabanov, ``Sparse divergence-free fourier
  approximations based on discrete $l^2$ projections.'' Unpublished manuscript.

\bibitem{breiman2001}
L.~Breiman, ``Random forests,'' {\em Machine Learning}, vol.~45, pp.~5--32, oct
  2001.

\bibitem{kingma2017adam}
D.~P. Kingma and J.~Ba, ``Adam: A method for stochastic optimization,'' 2017.

\end{thebibliography}
\bibliographystyle{ieeetr}
\end{document}